\documentclass[12pt]{amsart}
\usepackage{amsmath,amsthm,amssymb}
\usepackage{wasysym}
\usepackage{graphicx}
\usepackage{cite,float}
\usepackage{enumerate,amsmath,amsthm,amssymb,cite,mathrsfs,eufrak,graphicx,caption,subfigure}

\usepackage{tabularx}
\usepackage{tabulary}
\usepackage{ctable}
\newcolumntype{Q}{>{$\displaystyle}l<{$}}
\newcolumntype{A}{>{$}c<{$}}

\newcommand{\lemref}[1]{Lemma~\ref{#1}}

\newcommand{\disc}{\mathbb{D}}
\newcommand{\Scar}{\mathcal{S}^*_{car}}
\DeclareMathOperator{\RE}{Re}

\numberwithin{equation}{section}
\newtheorem{theorem}{Theorem}[section]
\newtheorem{lemma}[theorem]{Lemma}
\newtheorem{corollary}[theorem]{Corollary}
\theoremstyle{remark}

\makeatletter
\@namedef{subjclassname@2020}{%
  \textup{2020} Mathematics Subject Classification}
\makeatother

\begin{document}
\title[Inclusion Relations and Radius Problems]{Inclusion Relations and Radius Problems for a subclass of starlike functions}
 	
\author[P. Gupta]{Prachi Gupta}
	\address{Department of Mathematics, University of Delhi, Delhi--110 007, India}
 	\email{prachigupta161@gmail.com}

 	\author[S. Nagpal]{Sumit Nagpal}
 	\address{Department of Mathematics, Ramanujan College, University of Delhi,
Delhi--110 019, India}
 	\email{sumitnagpal.du@gmail.com }
 	
 	\author[V. Ravichandran]{V. Ravichandran}
 	
 	\address{Department of Mathematics, National Institute of Technology, Tiruchirappalli--620 015, India}
 	\email{vravi68@gmail.com, ravic@nitt.edu}
 	\dedicatory{Dedicated to Prof.\@ Ajay Kumar on the occasion of his retirement}

 	\begin{abstract}
By considering the polynomial function $\phi_{car}(z)=1+z+z^2/2,$ we define the class $\Scar$ consisting of normalized analytic functions $f$ such that $zf'/f$ is subordinate to $\phi_{car}$ in the unit disk. The inclusion relations and various radii constants associated with the class $\Scar$ and its connection with several well-known subclasses of starlike functions is established. As an application, the obtained results are applied to derive the properties of the partial sums and convolution.
 	\end{abstract}
 	
 	\maketitle
 	
\section{Introduction}
For $r>0$, let $\mathbb{D}_r:=\{z\in \mathbb{C}:|z|<r\}$. Let $\mathcal{A}$ denote the class of all analytic functions $f$ defined in the unit disk $\mathbb{D}:=\mathbb{D}_1$ such that $f(0)=f'(0)-1=0$ and $\mathcal{S}$ be its subclass consisting of univalent functions. Let $\mathcal{S}^*$ and $\mathcal{K}$ denote the  well-known subclasses of $\mathcal{S}$ consisting of starlike and convex functions respectively. Another interesting subclass of $\mathcal{A}$ is the class $\mathcal{C}$ of close-to-convex functions introduced by Kalpan \cite{KAP}. If $f \in \mathcal{C}$, then there exists a convex function $g$ (not necessarily normalized) such that $\RE (f'(z)/g'(z))>0$ for all $z\in \disc$. By Noshiro-Warschawski Theorem, $\mathcal{C}\subset \mathcal{S}$. Moreover, every starlike function is close-to-convex, but the converse need not be true in general. For example, the polynomials
\[f_n(z)=z+\frac{z^2}{2}+\frac{z^3}{3}+\cdots+\frac{z^n}{n},\quad n\geq1\]
are close-to-convex, but $f_n(\disc)$ is a non-starlike domain for $n\geq3$. The function
\begin{equation}\label{eq1.1}
\phi_{car}(z):=1+f_2(z)=1+z+\frac{z^2}{2},\quad z\in\disc
\end{equation}
maps $\disc$ onto a domain $\Omega_{car}$ bounded by the cardioid $(4x^2+4y^2-8x-1)^2+4(4x^2+4y^2-12x+1)=0$. To see this, note that $\phi_{car}(e^{it})=x+iy$ where
\[x=1+\cos t+\frac{1}{2}\cos 2t \quad \mbox{and}\quad y=\sin t+\frac{1}{2}\sin 2t\]
so that $4x^2+4y^2-8x-1=4 \cos t$ and $4x^2+4y^2-12x+1=-4\cos^2 t$. Note that $\phi_{car}(\disc)=\Omega_{car}$ lies in the right half-plane, is starlike with respect to $\phi_{car}(0)=1$ and symmetric with respect to the real axis with $\phi_{car}'(0)>0$.

Let $\mathcal{S}^*(\varphi)$ denotes the unified class of starlike functions introduced and studied by Ma and Minda\cite{MaMinda} consisting of functions $f\in \mathcal{S}$ satisfying $zf'(z)/f(z)\prec \varphi(z)$ for all $z\in \disc$ where $\varphi$ is a univalent function with positive real part, maps $\mathbb{D}$ onto a domain symmetric with respect to the real axis and starlike with respect to $\varphi(0)=1$ and $\varphi'(0)>0$. Since the function $\phi_{car}$ satisfy all the pre-requisites, therefore the class $\Scar:=\mathcal{S}^*(\phi_{car})$ is well-defined. It is easy to see that a function $f\in \Scar$ if and only if there exists an analytic function $q$, $q \prec \phi_{car}$ such that
\[f(z)=z\exp\left(\int_{0}^{z}\frac{q(t)-1}{t}dt\right).\]
By taking $q(z)=\phi_{car}(z)=1+z+z^2/2, $ we obtain
 \begin{equation}\label{fcar}
 	f_{car}(z):=z\exp\left(z+\frac{z^2}{4}\right)=z+z^2+\frac{3z^3}{4}+\frac{5z^4}{12}+\frac{19z^5}{96}+\cdots
 \end{equation}
Using the results obtained by Ma and Minda \cite{MaMinda}, it follows that the subordinations $f(z)/z\prec f_{car}(z)/z$ and $zf'(z)/f(z)\prec zf'_{car}(z)/f_{car}(z)$ holds in $\disc$ for any $f\in \Scar$. Also, the inequality $-f_{car}(-|z|)\leq|f(z)|\leq f_{car}(|z|)$ is valid for all $z\in \disc$ and for every $f\in\Scar$, with equality holding for some non-zero $z$ if and only if $f$ is a rotation of $f_{car}$. In particular, $f(\disc)$ contains the disk $\left\{w:|w|<-f_{car}(-1)=e^{-3/4}\approx 0.47236\right\}.$ Also, if $f\in \Scar$ is given by the Taylor series representation
\begin{equation}\label{eq}
 f(z)=z+\sum_{n=2}^{\infty}a_n z^n
\end{equation}
then $|a_2|\leq 1,$ $|a_3|\leq 3/4$ and $|a_4|\leq 5/12$ by invoking \cite[Theorem 1, p. 38]{ALI1}. All these bounds are sharp for the function $f_{car}$ given by \eqref{fcar}. The main aim of the paper is to establish a linkage of the class $\Scar$ with other subclasses of starlike  functions through inclusion relations and radii constants. As an application, these results are applied to derive the properties of the partial sums and convolution.


In Section 2, few lemmas have been proved that will be needed in the study of the class $\Scar$ in the subsequent sections. We evaluate the maximum and minimum of $\RE (\phi_{car}(z))$ over the circle $|z|=r$. Also, the radii $r_a$ and $R_a$ are computed so that the inclusion relation $\{w:|w-a|<r_a\}\subseteq \Omega_{car}\subseteq\{w:|w-a|<R_a\}$ is satisfied where $\phi_{car}(-1)<a<\phi_{car}(1)$. Also, a condition is determined under which the class $\Scar$ is closed under convolution with convex functions. Section 3 is devoted to establish the inclusion relations of the class $\Scar$ with various subclasses of starlike functions introduced in the literature. In addition, a coefficient inequality is derived to investigate the properties of second partial sum of a function $f\in \mathcal{A}$ belonging to well-known classes of univalent functions. Several radii constants associated with the class $\Scar$ are computed in Sections 4 and 5 of the paper. All the obtained results are best possible.

Let $\chi \in \mathcal{A}$ and $\mathcal{F}_i^\chi$ ($i=1,2,3$) be the classes of analytic functions defined as follows:
\[\mathcal{F}_1^\chi=\left\{f\in \mathcal{A}: \RE\left(\frac{f(z)}{g(z)}\right)>0 \mbox{ for some }g\in \mathcal{A} \mbox{ with }  \RE\left(\frac{g(z)}{\chi(z)}\right)>0\right\}\]
\[\mathcal{F}_2^\chi=\left\{f\in \mathcal{A}: \left|\frac{f(z)}{g(z)}-1\right|<1 \mbox{ for some }g\in \mathcal{A} \mbox{ with }  \RE\left(\frac{g(z)}{\chi(z)}\right)>0\right\}\]
\[\mathcal{F}_3^\chi=\left\{f\in \mathcal{A}: \RE\left(\frac{f(z)}{\chi(z)}\right)>0\right\}\]
Motivated by MacGregor \cite{MAC1, MAC2}, several authors have determined various radii constants for these classes for several choices of $\chi$ like $z$ \cite{ALI, MAC}, $z+z^2/2$ \cite{KANAGA}, $z/(1-z^2)$ \cite{KHATTER2}, $z/(1-z)^2$ \cite{AMVS} and $z/(1+z)$ \cite{ASHA}. In the last section of this paper, we determine the $\Scar$-radius for the classes $\mathcal{F}_i^\chi$ ($i=1,2,3$) for these choices of $\chi$.

 \section{Preliminary Lemmas}
In this section, we will prove some lemmas which will be needed in our further investigation. The following lemma is useful while computing the upper and lower bounds of $\RE(zf'(z)/f(z))$ for a function $f\in \Scar$.  	
\begin{lemma}\label{result1}
If $0<r<1,$ then the function $\phi_{car}$ defined by \eqref{eq1.1} satisfies	
	\[\underset{|z|=r}{\min}\RE(\phi_{car}(z))=\begin{cases}
	1-r+r^2/2,& 0<r\leq 1/2\\
	(3-2r^2)/4,& \displaystyle1/2\leq r<1
	\end{cases}\]	
	and
	\[\underset{|z|=r}{\max}\RE(\phi_{car}(z))=1+r+\frac{r^2}{2}=\phi_{car}(r).\]	
\end{lemma}
\begin{proof}
	For $z=re^{it} $ $(0\leq t\leq 2\pi),$ we have
	\[\RE(\phi_{car}(re^{it}))=1+r\cos t+\frac{r^2}{2}\cos 2t=1+rx+r^2x^2-\frac{r^2}{2}:=G(x)\]
	where $x=\cos t,$ $-1\leq x\leq1$. Note that $G'(x)=r+2r^2x$, $G''(x)=2r^2$ and $G'(-1/(2r))=0$. It is easily seen that $x_0=-1/(2r)\in [-1,1]$ if and only if $r\geq 1/2.$ Consider the following two cases:

	Case 1: For $0<r<1/2,$ $G'(x)>0$ and hence is an increasing function. Consequently $\max\{\RE(\phi_{car}(z)):|z|=r\}=G(1)=1+r+r^2/2$
	and $\min\{\RE(\phi_{car}(z)):|z|=r\}=G(-1)=1-r+r^2/2$.
	
Case 2: For $1/2\leq r<1$, $x_0=-1/(2r)$ is a point of relative minima as $G''(x_0)>0$. Also, $G(1)>G(-1)$ so that $\max \{\RE(\phi_{car}(z)):|z|=r\}=G(1)=1+r+r^2/2$ and $\min\{\RE(\phi_{car}(z)):|z|=r\}=G(x_0)=(3-2r^2)/4$.
\end{proof}

The next result is the main key lemma for determining the various inclusion results and result constants concerning the class $\Scar$.
\begin{lemma}\label{result7}
	For $1/2<a<5/2,$ let $r_a$ be given by
	\begin{align*}
	r_a=\begin{cases}
	(2a-1)/2, & 1/2<a\leq \displaystyle3/2\\
	(5-2a)/2,& 3/2\leq a< 5/2
	\end{cases}
	\end{align*}
	and $R_a$ be given by
	\begin{align*}
	R_a=\begin{cases}
	(5-2a)/2,&1/2<a\displaystyle\leq 7/6\\
	\sqrt{(2a-1)^3/(8(a-1))},& 7/6\leq a<5/2.
	\end{cases}	
	\end{align*}
	Then \[\{w:|w-a|<r_a\}\subseteq \Omega_{car}\subseteq\{w:|w-a|<R_a\} \]
\end{lemma}
\begin{proof}
For $t\in [0,2\pi]$, if we let $\cos t=x$, $-1\leq x\leq 1$, then the square of the distance from the point $(a,0)$ to the points on the cardioid $\phi_{car}(e^{it})$ is given by
\[g(x)=a^2-a+\frac{5}{4}+(3-2a)x+2(1-a)x^2.\]
Note that $g'(x)=3-2a+4(1-a)x$ and $g'(x_0)=0$ for $x_0=(3-2a)/(4(a-1))$. Also, $ x_0\in [-1,1]$ if and only if $a\geq 7/6.$  Moreover, $g''(x)=4(1-a)$ and $g(1)-g(-1)=6-4a.$ Consider the following two cases:

Case 1: For $1/2<a\leq 7/6,$ $g'(x)>0$ so that $g$ is an increasing function and hence $\min\{g(x):x\in [-1,1]\}=g(-1)$ and $\max\{g(x):x\in [-1,1]\}=g(1)$.
	
Case 2:	For $7/6\leq a < 5/2$, then $g''(x_0)<0$ and hence $g$ attains its relative maxima at $x_0$. If $7/6\leq a\leq 3/2$, then $g(1)\geq g(-1)$ and $\min\{g(x):x\in [-1,1]\}=g(-1)$. If $3/2\leq a<5/2$, then $g(-1)\geq g(1)$ and $\min\{g(x):x\in [-1,1]\}=g(1)$. In both the cases, $\max\{g(x):x\in [-1,1]\}=g(x_0)$.

Since $g(-1)=(a-1/2)^2$, $g(1)=(5-2a)^2/4$ and $g(x_0)=(2a-1)^3/(8(a-1))$, therefore the proof of the theorem is complete by observing that $r_a=\sqrt{ \min\{g(x):x\in [-1,1]\}}$ and $R_a=\sqrt{\max\{g(x):x\in [-1,1]\}}$. The cases $a=1$ ($r_a=1/2$, $R_a=3/2$) and $a=2$ ($r_a=1/2$, $R_a=\sqrt{27/8}$) of this lemma are illustrated in Figure \ref{lemma}.
\end{proof}

\begin{figure}[h]
	\begin{center}
		\subfigure[$a=1$]{\includegraphics[width=2in]{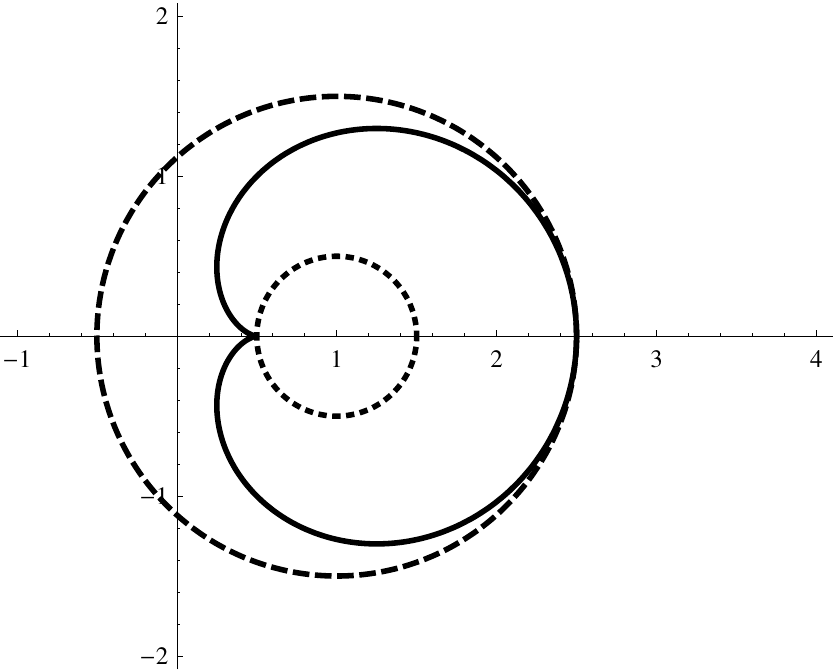}}\hspace{10pt}
		\subfigure[$a=2$]{\includegraphics[width=2in]{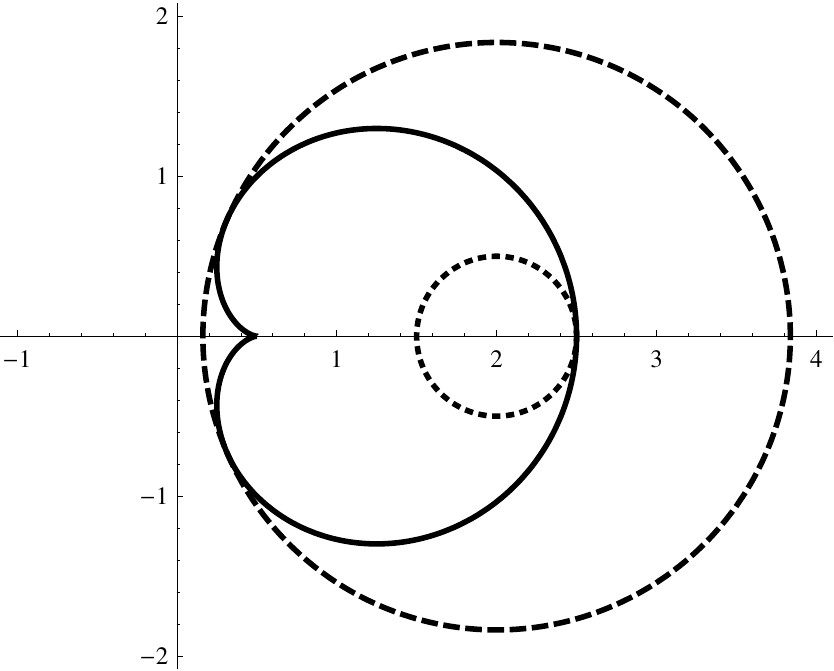}}\hspace{10pt}
		\caption{Illustration of Lemma \ref{result7}.}\label{lemma}
	\end{center}
\end{figure}

By \cite[Theorem 5, p.\ 167]{MaMinda}, if $\varphi(\mathbb{D})$ is convex, $f \in\mathcal{S}^*(\varphi)$ and $g \in \mathcal{K}$, then the convolution $f*g \in \mathcal{S}^*(\varphi)$. Note that
\[\RE\left(1+\frac{\phi_{car}''(z)}{\phi_{car}'(z)}\right)=1+\RE\left(\frac{z}{1+z}\right)=1-\frac{|z|}{1-|z|}\]
which is clearly positive if $|z|<1/2$, where $\phi_{car}$ is given by \eqref{eq1.1}. This shows that $\phi_{car}(\mathbb{D}_{1/2})$ is convex. Therefore, we have obtained the following result.

\begin{lemma}\label{result10}
Let $f\in\Scar$ and $g\in\mathcal{K}.$ Then the function $(f\ast g)(\rho z)/\rho\in\Scar$ for $0< \rho\leq 1/2.$
\end{lemma}

\section{Inclusion Relations}
For specific choices of $\varphi$, the class $\mathcal{S}^*(\varphi)$ reduces to several interesting subclasses of $\mathcal{S}^*$. To find the inclusion relations, firstly consider those subclasses which involve a parameter. For instance, $\mathcal{S}^*[A,B]:=\mathcal{S}^*((1+Az)/(1+Bz))$ is the Janowski class of starlike functions ($-1\leq B<A\leq1$), $\mathcal{S}^*(\alpha):=\mathcal{S}^*[1-2\alpha,-1]$ is the class of starlike functions of order $\alpha$ $(0\leq \alpha<1)$, and $\mathcal{S}\mathcal{S}^*(\beta):=\mathcal{S}^*(((1+z)/(1-z))^\beta)$ is the class of strongly starlike functions of order $\beta$ ($0<\beta\leq 1$). For $k\geq 0$, the class $k-\mathcal{S}^*$ of $k$-starlike functions introduced by Kanas and Wisniowska \cite{KANAS} consists of functions $f \in \mathcal{A}$ satisfying
\[\RE \left(\frac{zf'(z)}{f(z)}\right)>k\left|\frac{zf'(z)}{f(z)}-1\right|\quad (z\in \mathbb{D}).\]
In fact, $k-\mathcal{S}^*=\mathcal{S}^*(p_k)$ where $p_k$ is defined in \cite[p.\ 331-333]{KANAS}. For $0\leq \alpha<1$, the classes $\mathcal{S}^*_{e}(\alpha)=\mathcal{S}^*(\alpha+(1-\alpha)e^z)$ and $\mathcal{S}^*_{L}(\alpha)=\mathcal{S}^*(\alpha+(1-\alpha)\sqrt{1+z})$ were introduced by Khatter \emph{et al.} \cite{KHATTER} which reduces to the classes $\mathcal{S}^*_{e}$ \cite{MEND1} and $\mathcal{S}^*_{L}$ \cite{SOKOL} respectively for $\alpha=0$. Moreover, let $\mathcal{S}^*(\sqrt{1+cz})$ ($0<c\leq 1$) be the subclass of starlike functions introduced by Sok\'{o}\l \cite{SOKOL1} associated with right loop of the Cassinian ovals $(u^2+v^2)^2-2(u^2-v^2)=c^2-1$, which also reduces to the class $\mathcal{S}^*_{L}$ for $c=1$. For more details and generalizations, see \cite{ADAM, ALI3, SIM, LECKO, DENIZ, DENIZ1}. The first theorem of this section investigates the inclusion relations between these classes with the class $\Scar$.

\begin{theorem}\label{th4.1}
 		The class $\Scar$ satisfies the following relationships:
 		\begin{enumerate}
 			\item [(i)] $\Scar\subset\mathcal{S}^*(\alpha)$ for $0\leq\alpha\leq 1/4;$
 			\item[(ii)] $\Scar\subset\mathcal{S}\mathcal{S}^*(\beta)$ for $\beta_0\leq \beta\leq 1,$ where $\beta_0=(2/\pi)\tan^{-1}(3\sqrt{3/5}) \approx 0.743253$;
 			\item[(iii)] $k-\mathcal{S}^*\subset\Scar$ for $k\geq 5/3$;
 			\item[(iv)] $\mathcal{S}^*_{e}(\alpha)\subset\mathcal{S}^*_{car}$ for $\alpha_0\leq\alpha<1$, where $\alpha_0=(e-2)/(2(e-1))\approx 0.209011$;
 			\item[(v)] $\mathcal{S}^*_L(\alpha)\subset \mathcal{S}^*_{car}$ for $1/2\leq\alpha< 1$;
 \item[(vi)] $\mathcal{S}^*(\sqrt{1+cz})\subset \mathcal{S}^*_{car}$ for $0<c\leq 3/4$;
 \item[(vii)] $\mathcal{S}^*[A,B]\subset\Scar$, where $-1< B<A\leq1$ if one of the following conditions holds:
	\begin{enumerate}
		\item [(a)] $1-B^2< 2(1-AB)\leq 3(1-B^2)$ and $ 2A\leq 1+B,$
		\item[(b)] $3(1-B^2)\leq 2(1-AB)<5(1-B^2)$ and $2A\leq 3+5B;$
	\end{enumerate}
\item[(viii)] $\Scar\subset\mathcal{S}^*[1,-(M-1)/M]$ for $M\geq (3+\sqrt{5})/4\approx 1.309$;
 		\end{enumerate}
 	\end{theorem}
 \begin{proof}
For (i) and (ii), let $f\in \Scar$. Then $zf'(z)/f(z)\prec \phi_{car}(z)$, where $\phi_{car}$ is given by \eqref{eq1.1} and $\phi_{car}(\mathbb{D})=\Omega_{car}$. By Lemma \ref{result1}, we have
 		\[\RE\left(\frac{zf'(z)}{f(z)}\right)>\underset{|z|=1}{\min}\RE(\phi_{car}(z))=\frac{1}{4}\]
so that $\Scar\subset\mathcal{S}^*(1/4).$ Also
\begin{align*}
 		\left|\arg\left(\frac{zf'(z)}{f(z)}\right)\right|&<\underset{|z|=1}{\max}\arg(\phi_{car}(z))= \underset{t\in(-\pi,\pi]}{\max}\arg(\phi_{car}(e^{it}))\\
 		&=\max_{t\in(-\pi,\pi]} \tan^{-1}(g(t))=\tan^{-1}(\max_{t\in(-\pi,\pi]}g(t))
 		\end{align*}
 where $g(t)=2\sin t(1+\cos t)/(1+2\cos t+2\cos^2 t)$. A simple calculation shows that $g$ attains its maximum value at $t_0=\cos^{-1}(-1/4)$ and
 		\[\left|\arg\left(\frac{zf'(z)}{f(z)}\right)\right|<\tan^{-1}(g(t_0))=\tan^{-1}\left(3\sqrt{\frac{3}{5}}\right).\]
Consequently $|\arg(zf'(z)/f(z))|\leq \beta_0\pi/2$ where $\beta_0=(2/\pi)\tan^{-1}(3\sqrt{3}/5)$ and hence $f \in \mathcal{SS}^*(\beta_0)$. The curves $\gamma_1: \RE w=1/4$ and $\gamma_2: \arg w=\tan^{-1}(3\sqrt{3}/5)$ are depicted in Figure \ref{inclusion}((a) and (b)) which shows that the constants $1/4$ and $\beta_0$ are best possible.

For (iii), let $f\in k-\mathcal{S}^*$. Then the quantity $zf'(z)/f(z)$ lies inside the domain $\Gamma_k=\{w:\mathbb{C}:\RE w>k|w-1|\}$ which represents the whole right half-plane for $k=0$, a hyperbolic region for $0<k<1$, a parabolic region for $k=1$ and an elliptic region for $k>1$. For $k>1$, the boundary of $\Gamma_k$ is an ellipse with the equation
\[\frac{(x-\lambda)^2}{a^2}+\frac{y^2}{b^2}=1,\]
where $\lambda=k^2/(k^2-1)$, $a=k/(k^2-1)$ and $b=1/\sqrt{k^2-1}$. This ellipse lies inside $\Omega_{car}$ if $\lambda-a\geq 1/2$ and $\lambda+a\leq 5/2$. As $k>1$, the former condition $\lambda-a=k/(k+1)>1/2$ is obviously satisfied, while the later condition $\lambda+a=k/(k-1)\leq 5/2$ yields $k\geq 5/3$. The domain $\Gamma_{5/3}$ lies completely inside $\Omega_{car}$ which is evident from Figure \ref{inclusion}(c) in which the boundary $\partial \Gamma_{5/3}=\gamma_3$ is illustrated. Moreover, as $k$ increases, the domains $\Gamma_k$ keeps on shrinking, therefore $f\in \Scar$ for $k\geq 5/3$ and this constant is best possible.

\begin{figure}[h]
	\begin{center}
		\subfigure[$\gamma_1$]{\includegraphics[width=1.2in]{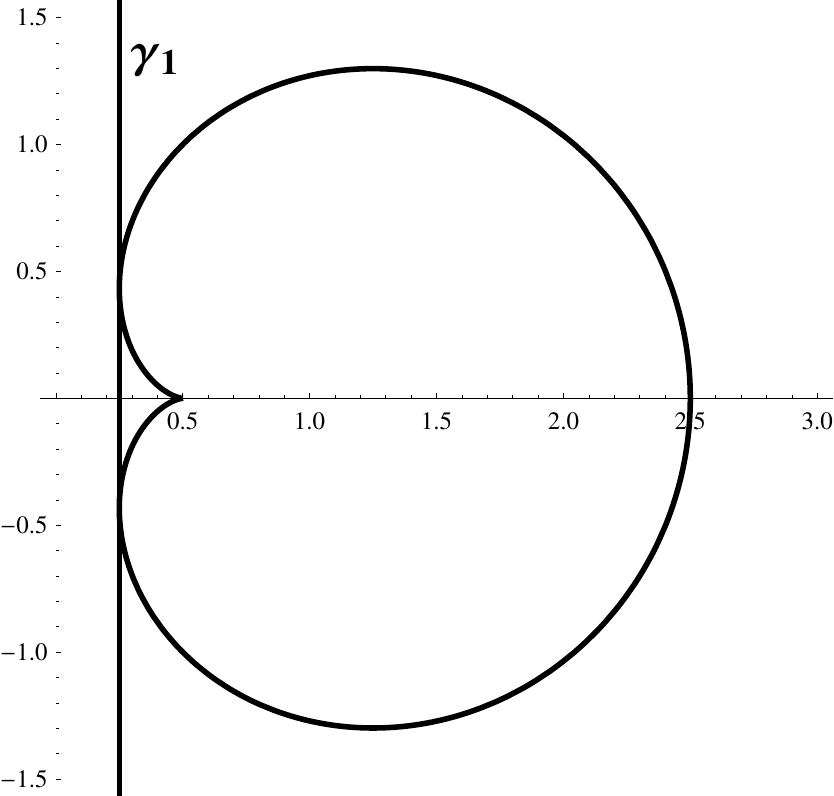}}\hspace{10pt}
		\subfigure[$\gamma_2$]{\includegraphics[width=1.2in]{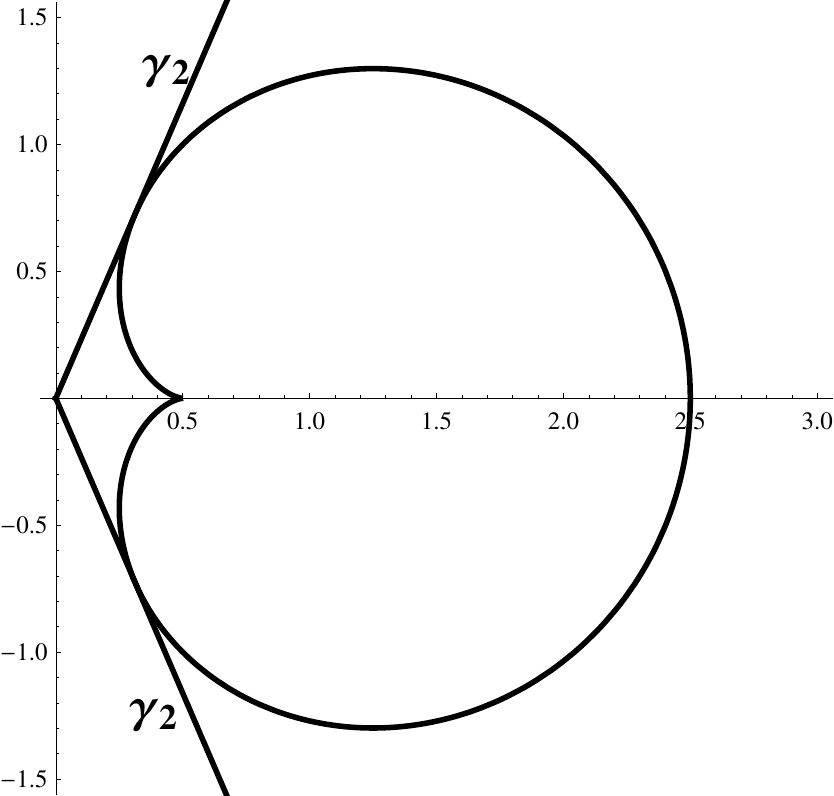}}\hspace{10pt}
		\subfigure[$\gamma_3$]{\includegraphics[width=1.2in]{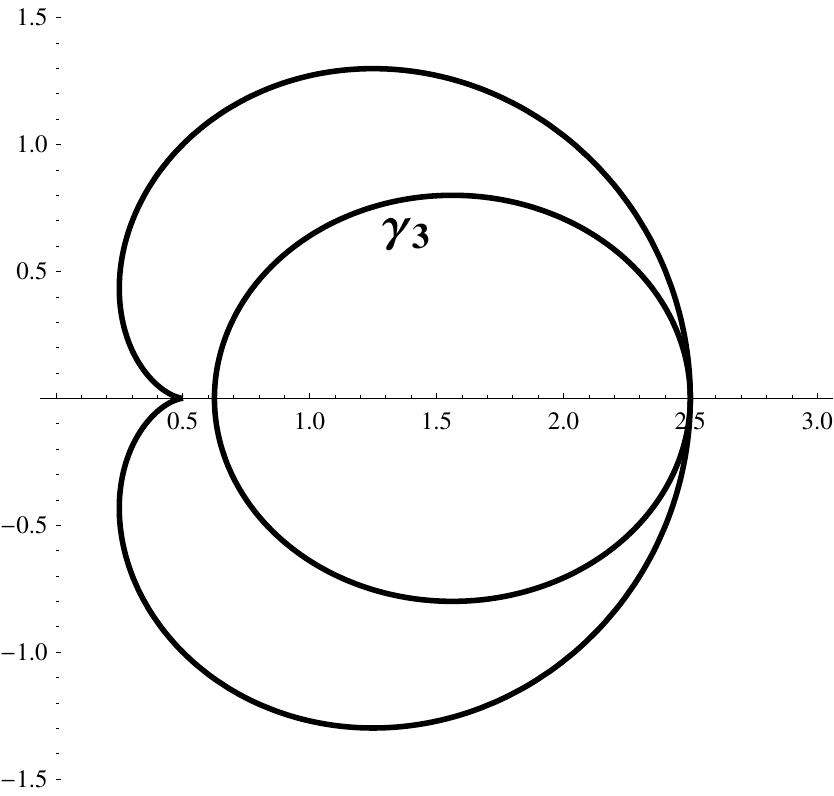}}\hspace{10pt}
\subfigure[$\gamma_4$]{\includegraphics[width=1.2in]{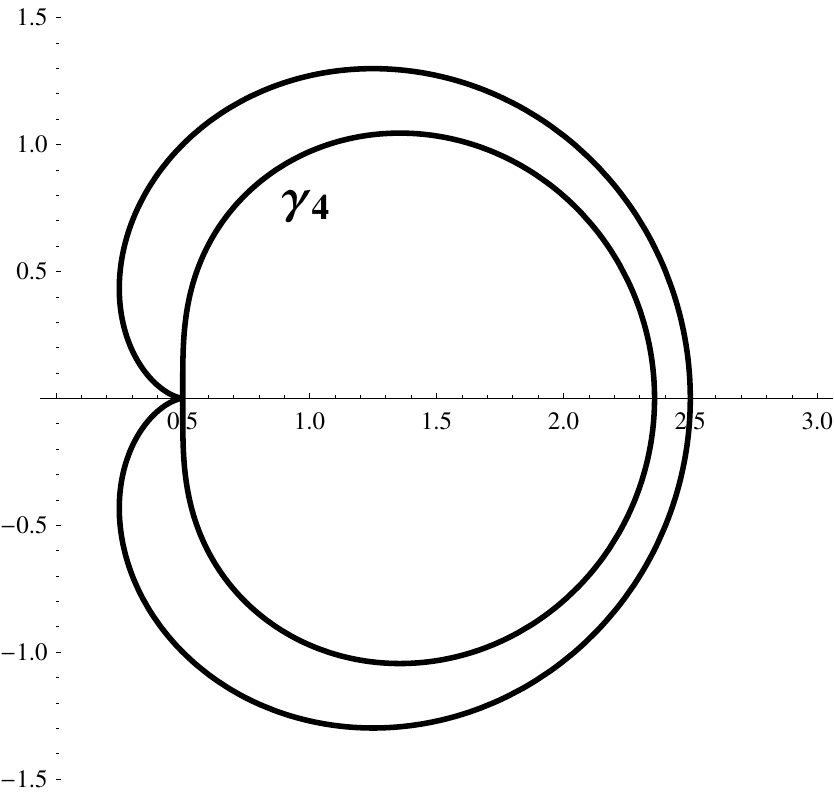}}\hspace{10pt}
\subfigure[$\gamma_5$]{\includegraphics[width=1.2in]{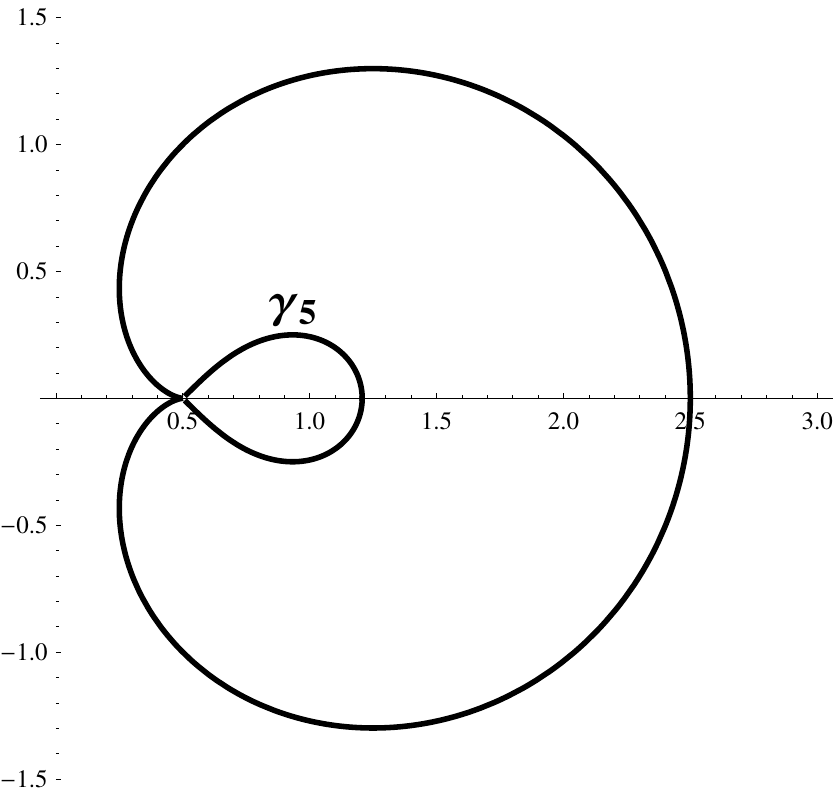}}\hspace{10pt}
\subfigure[$\gamma_6$]{\includegraphics[width=1.2in]{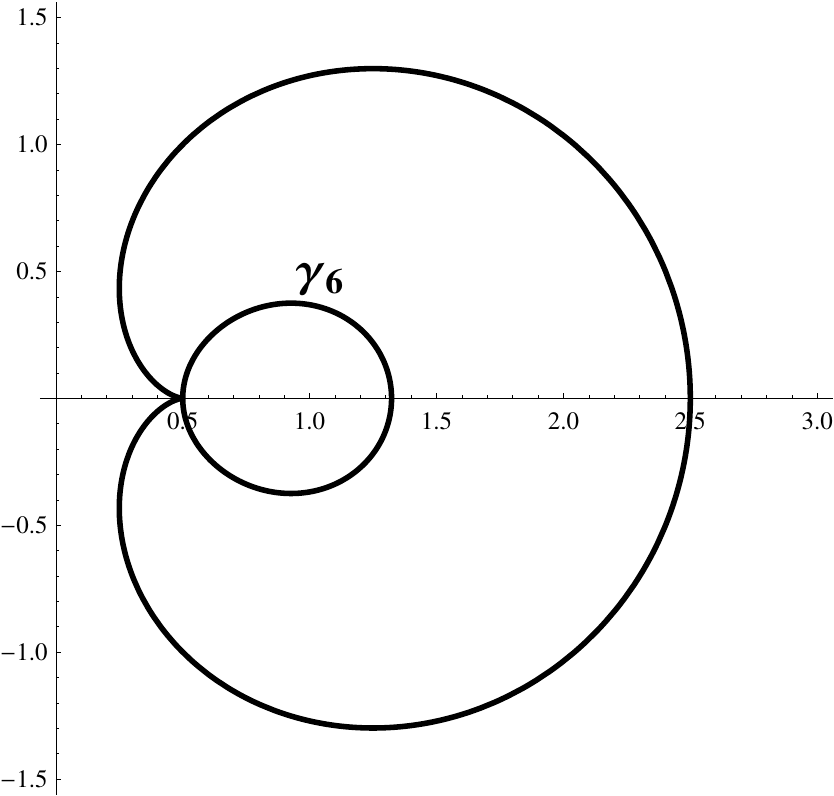}}\hspace{10pt}
\subfigure[$\gamma_7$]{\includegraphics[width=1.2in]{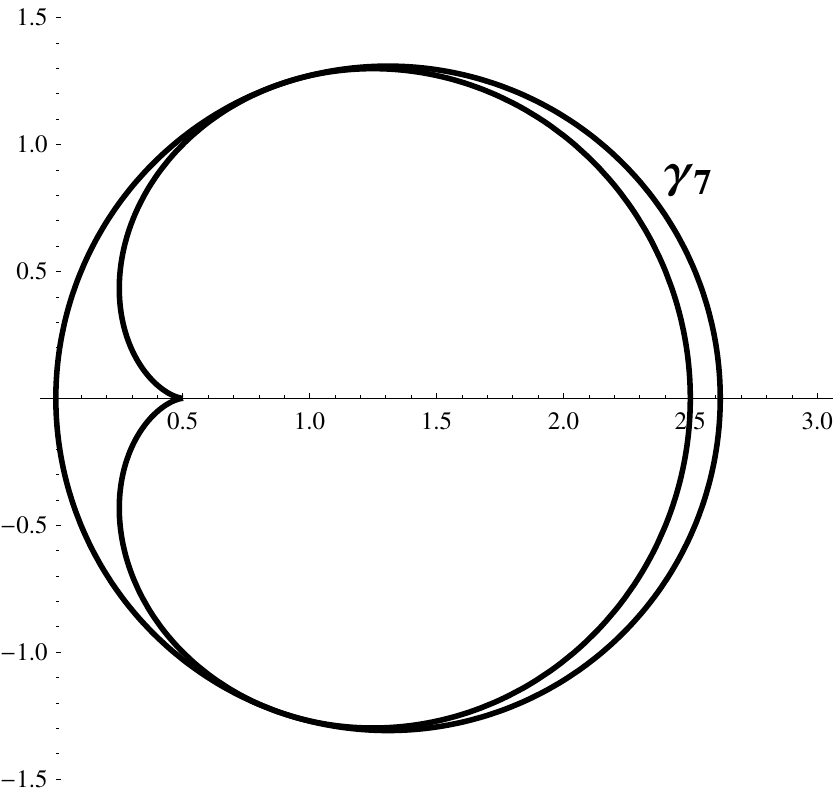}}\hspace{10pt}
		\caption{Inclusion Relations.}\label{inclusion}
	\end{center}
\end{figure}

For the proof of (iv), note that if $f \in\mathcal{S}^*_{e}(\alpha)$, then $zf'(z)/f(z)\prec \psi_\alpha(z)$ where $\psi_\alpha(z)=\alpha+(1-\alpha)e^z$ and $zf'(z)/f(z)$ lies inside the domain $D_\alpha=\{w:|\log((w-\alpha)/(1-\alpha))|<1\}$. By \cite[Lemma 2.1, p.\ 236]{KHATTER}, we have
\[\alpha+(1-\alpha)e^{-1}=\underset{|z|=1}{\min}\RE(\psi_\alpha(z))<\RE\left(\frac{zf'(z)}{f(z)}\right)<\underset{|z|=1}{\max}\RE(\psi_\alpha(z))=\alpha+(1-\alpha)e.\]
The condition $\alpha+(1-\alpha)e\leq 5/2$ gives $\alpha\geq (2e-5)/(2(e-1))\approx 0.127$ and the condition $\alpha+(1-\alpha)e^{-1}\geq 1/2$ provides $\alpha\geq(e-2)/(2(e-1))\approx 0.209.$ Therefore both the conditions are simultaneously satisfied for $\alpha\geq (e-2)/(2(e-1))=\alpha_0$. Moreover, $D_\alpha\subset D_{\alpha_0}\subset \Omega_{car}$ for $\alpha\geq \alpha_0$ (see Figure \ref{inclusion}(d), where $\gamma_4=\partial D_{\alpha_0}$). Thus $f \in \Scar$ for all $\alpha_0\leq \alpha<1$ and the bound $\alpha_0$ is best possible.

To prove (v), let $f\in\mathcal{S}^*_L(\alpha)$. Then $zf'(z)/f(z)\prec \nu_\alpha(z)$, where $\nu_\alpha(z)=\alpha+(1-\alpha)\sqrt{1+z}$. Also, the quantity $zf'(z)/f(z)$ lies inside the domain $G_\alpha=\{w:|((w-\alpha)/(1-\alpha))^2-1|<1\}$. Again \cite[Lemma 2.1, p.\ 236]{KHATTER} shows that
\[\alpha=\underset{|z|=1}{\min}\RE(\nu_\alpha(z))<\RE\left(\frac{zf'(z)}{f(z)}\right)<\underset{|z|=1}{\max}\RE(\nu_\alpha(z))=\alpha+(1-\alpha)\sqrt{2}.\]
It is clear that $\alpha+(1-\alpha)\sqrt{2}<1+\sqrt{2}<5/2$ and we must have $\alpha\geq 1/2$. Also, by Figure \ref{inclusion}(e), it is clear that the domain $G_{1/2}$ lies inside $\Omega_{car}$, where $\gamma_5=\partial G_{1/2}$. Also, $G_{\alpha_1}\subset G_{\alpha_2}$ for $\alpha_1\geq \alpha_2$. This forces us to conclude that $f \in \Scar$ for $1/2\leq \alpha< 1$ with $1/2$ as the best possible bound.

For (vi), let $f \in\mathcal{S}^*(\sqrt{1+cz})$. Then the quantity $zf'(z)/f(z)$ lies inside the domain $H_\alpha=\{w:|w^2-1|<c,\RE w>0\}$ and
\[\sqrt{1-c}<\RE\left(\frac{zf'(z)}{f(z)}\right)<\sqrt{1+c}.\]
Note that $\sqrt{1+c}\leq \sqrt{2}<5/2$ and $\sqrt{1-c}\geq 1/2$ provided $c\leq 3/4$. The similar reasoning shows that $f\in \Scar$ for $0<c\leq 3/4$ and this bound cannot be improved. Figure \ref{inclusion}(f) illustrates that the boundary $\partial H_{3/4}$ lies inside the cardioid.

For the part (vii), since $f\in\mathcal{S}^*[A,B]$, $zf'(z)f(z)\prec (1+Az)/(1+Bz)$. By \cite[Lemma 2.1, p.267]{RAVI}, it follows that
	\begin{align}\label{eqn11}
	\left|\frac{zf'(z)}{f(z)}-\frac{1-AB}{1-B^2}\right|\leq \frac{A-B}{1-B^2}
	\end{align}
which represents a disk with center $a:=(1-AB)/(1-B^2)$ and radius $r:=(A-B)/(1-B^2).$ The inequalities in part (a) are equivalent to $1/2<a\leq 3/2$ and $r\leq a-1/2$. Similarly, the inequalities in part (b) are equivalent to $3/2\leq a<5/2$ and $r\leq 5/2-a$. By Lemma \ref{result7}, the disk \eqref{eqn11} lies inside the domain $\Omega_{car}$ and hence $f \in \Scar$.

In the last part (viii), let $f \in \Scar$. Then $zf'(z)/f(z)\in \Omega_{car}$ for all $z\in \mathbb{D}$. Under the notations of Lemma \ref{result7}, $\Omega_{car}\subseteq\{w:|w-a|<R_a\}$ and note that $a=R_a$ if and only if $a=(3+\sqrt{5})/4$. These observations lead us to conclude that
\[\left|\frac{zf'(z)}{f(z)}-M_0\right|<M_0\]
where $M_0=(3+\sqrt{5})/4$. Thus $f\in \mathcal{S}^*[1,-(M-1)/M]$ for $M\geq M_0$ (see Figure \ref{inclusion}(g), $\gamma_7: |w-M_0|=M_0$).
\end{proof}

The particular choices of $A$ and $B$ in $\mathcal{S}^*[A,B]$ leads to several well-known classes of $\mathcal{S}^*$ which are extensively studied in literature. The class $\mathcal{S}^*[1,-(M-1)/M]$ ($M>1/2$) discussed in Theorem \ref{th4.1}(viii) was introduced by Janowski \cite{JANOWSKI1}. Similarly, the classes $\mathcal{S}^*[1-\alpha,0]$ $(0\leq \alpha<1)$ and $\mathcal{S}^*[\alpha,-\alpha]$ $(0<\alpha\leq 1)$ were introduced by Singh \cite{RAM} and Padmanabhan \cite{PADMANABHAN} respectively. The inclusion relations with these two classes and further applications are discussed in the next two corollaries.

\begin{corollary}\label{cor4.2}
$\mathcal{S}^*[1-\alpha,0]\subset\Scar$ for all $\alpha\in [1/2,1)$ and $\mathcal{S}^*[\alpha,-\alpha]\subset\Scar$ for all $\alpha\in (0,1/3]$. In particular, if $f\in \mathcal{A}$ satisfies
\[\left|\frac{zf'(z)}{f(z)}-1\right|<\frac{1}{2}\quad \mbox{or}\quad\left|\frac{zf'(z)/f(z)-1}{zf'(z)/f(z)+1}\right|<\frac{1}{3}\]
for all $z\in \mathbb{D}$, then $f \in \Scar$.
\end{corollary}

\begin{proof}
The inclusion relations are valid since $A$ and $B$ satisfy the condition (a) of Theorem \ref{th4.1}(vii).
\end{proof}

\begin{corollary}\label{cor}
(i) If $f \in \mathcal{A}$ is given by \eqref{eq} such that
\[\sum_{n=2}^\infty (2n-1)|a_n|\leq 1\]
then $f \in \Scar$. (ii) A function $f_n(z)=z+a_nz^n$ $(n=2,3,\ldots)$ belongs to the class $\Scar$ if and only if $|a_n|\leq1/(2n-1).$
\end{corollary}

\begin{proof}
If a function $f\in \mathcal{A}$ satisfies the given condition, then by the proof of \cite[Theorem 1, p.\ 110]{SILVERMAN}, it is evident that
\[\left|\frac{zf'(z)}{f(z)}-1\right|<\frac{1}{2}\]
for all $z\in \mathbb{D}$. Thus $f \in\Scar$ by Corollary \ref{cor4.2}. This proves part (i).

For the proof of (ii), suppose that the function $f_n(z)=z+a_nz^n\in \Scar$ $(n=2,3,\ldots)$. As $f_n\in \mathcal{S}^*$, $|a_n|\leq 1/n$ and the quantity $w=zf_n'(z)/f_n(z)=(1+na_nz^{n-1})/(1+a_nz^{n-1})$ maps $\disc$ onto the disk
\[\left|w-\frac{1-n|a_n|^2}{1-|a_n|^2}\right|< \frac{(n-1)|a_n|}{1-|a_n|^2}.\] Since $f_n\in\Scar$, this disk lies inside the domain $\Omega_{car}$. As $1-n|a_n|^2\leq 1-|a_n|^2$,  therefore by Lemma \ref{result1}, it follows that \[\frac{(n-1)|a_n|}{1-|a_n|^2}\leq \frac{1-n|a_n|^2}{1-|a_n|^2}-\frac{1}{2}.\] This gives $|a_n|\leq 1/(2n-1).$ The converse is obviously true by part (i) of the corollary.
\end{proof}

Let us now give an application of Corollary \ref{cor} to partial sums of analytic functions. Recall that if a function $f\in \mathcal{A}$ is given by \eqref{eq}, then the $n$th ($n\geq 1$) partial sum of $f$ is the polynomial $f_n(z)=z+a_2z^2+a_3z^3+\cdots+a_nz^n$. A survey on partial sums can be found in \cite{RAVI1}. The following result in this direction determines sharp radii constants for the second partial sum of a function $f\in \Scar$.

\begin{theorem}\label{th5.2}
Let $f \in \Scar$. Then the second partial sum $f_2$ is starlike in $|z|<1/2$ and convex in $|z|<1/4$. Also, $f_2(\rho z)/\rho\in \Scar$ for $0<\rho\leq 1/3$. All the constants are sharp.
\end{theorem}

\begin{proof}
Suppose that $f\in \Scar$ is given by \eqref{eq}. Then $|a_2|\leq1$ and the second partial sum is given by $f_2(z)=z+a_2z^2$. Note that
\[\RE \left(\frac{zf_2'(z)}{f_2(z)}\right)=\RE\left(1+\frac{a_2 z}{1+a_2 z}\right)\geq 1-\frac{|a_2||z|}{1-|a_2||z|}.\]
As a result, $f$ is starlike in $|z|<1/2$. Since $f_2$ is convex in $|z|<r$ if and only if $zf_2'$ is starlike in $|z|<r$, therefore it follows that $f_2$ is convex in $|z|<1/4$. The last part follows by using Corollary \ref{cor}. For sharpness, consider the function $f_{car}$ given by \eqref{fcar}. Its second partial sum is $\tau_2(z)=z+z^2$ which satisfies
\[\left.\frac{z\tau_2'(z)}{\tau_2(z)}\right|_{z=-1/2}=0, \quad \left.1+\frac{z\tau_2''(z)}{\tau_2'(z)}\right|_{z=-1/4}=0 \quad \mbox{and}\quad \left.\frac{z\tau_2'(z)}{\tau_2(z)}\right|_{z=-1/3}=\frac{1}{2}.\qedhere\]
\end{proof}

The next theorem determines the bounds on $\rho$ such that the second partial sum  $f_2(\rho z)/\rho$ belongs to $\Scar$ for a function $f$ belonging to several classes of univalent functions. The same technique can be applied to obtain the results for $f$ belonging to other classes as well.

\begin{theorem}
Let $f\in \mathcal{A}$ be given by \eqref{eq} and $f_2$ denotes its second partial sum. Then we have the following:
\begin{enumerate}
\item If $f\in \mathcal{K}$, then $f_2(\rho z)/\rho\in \Scar$ for $0<\rho\leq 1/3$;
\item If $f\in \mathcal{S}$ (or $\mathcal{S}^*$ or $\mathcal{C}$), then $f_2(\rho z)/\rho\in \Scar$ for $0<\rho\leq 1/6$.
\end{enumerate}
\end{theorem}

\begin{proof}
Since $f_2(\rho z)/\rho=z+a_2\rho z^2$, therefore if $3|a_2|\rho\leq 1$, then $f_2(\rho z)/\rho\in \Scar$ by Corollary \ref{cor}. If $f \in \mathcal{K}$, then $|a_2|\leq 1$ so that $3|a_2|\rho\leq 1$ for $0<\rho\leq 1/3$. The result is sharp for the half-plane mapping $l(z)=z/(1-z)$. Similarly, if $f \in \mathcal{S}$, then $|a_2|\leq 2$ which gives $3|a_2|\rho\leq 1$ for $0<\rho\leq 1/6$ and in this case, Koebe function $k(z)=z/(1-z)^2$ verifies the sharpness of the result.
\end{proof}

\begin{figure}[h]
	\begin{center}
		\subfigure[$2/(1+e^{-z})$]{\includegraphics[width=1.8in]{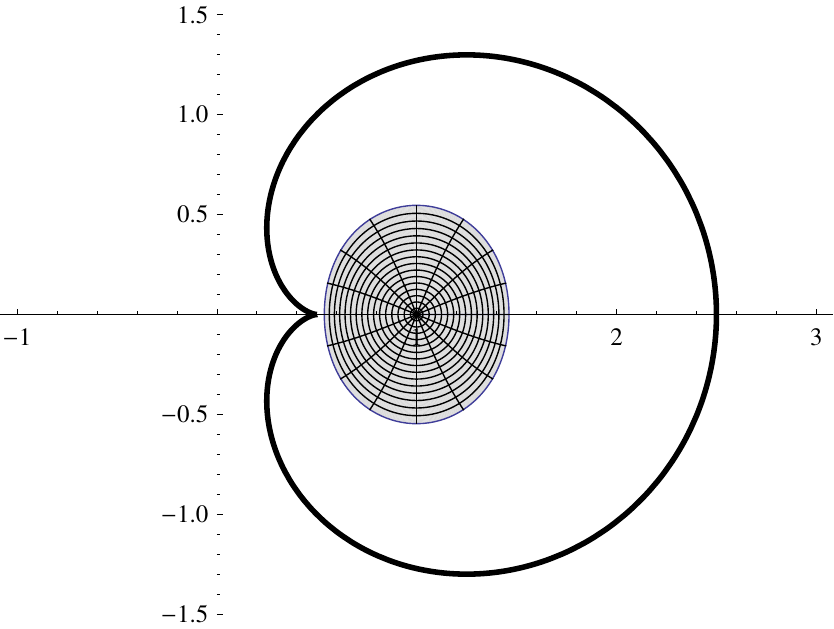}}\hspace{10pt}
		\subfigure[$\cosh z$]{\includegraphics[width=1.8in]{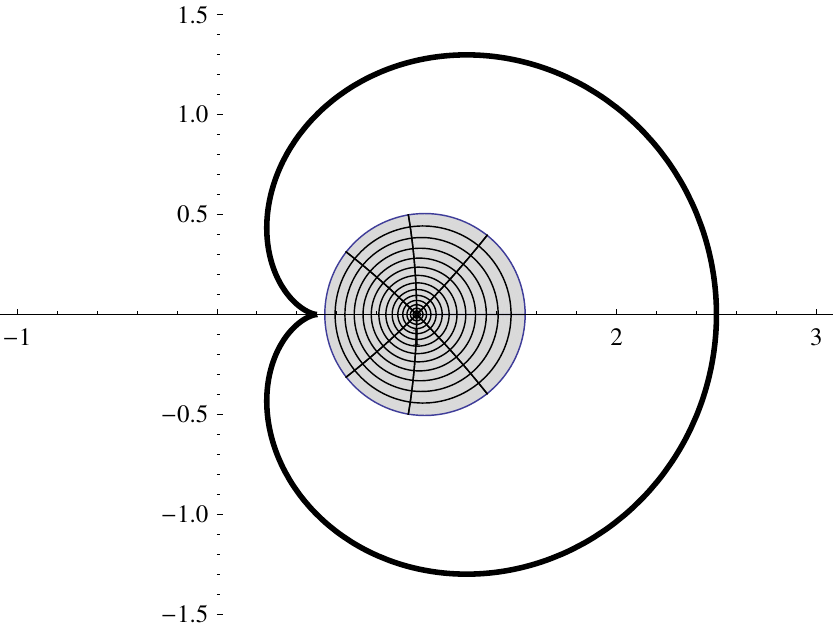}}\hspace{10pt}
        \subfigure[$1+\frac{z}{k}\left(\frac{k+z}{k-z}\right)$]{\includegraphics[width=1.8in]{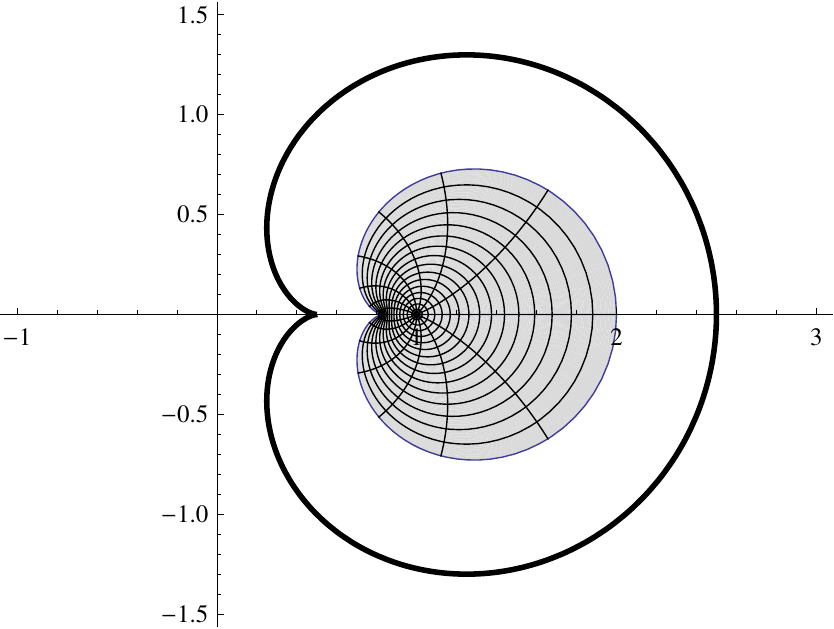}}
		\caption{$\Scar$-radius is unity.}\label{subset}
	\end{center}
\end{figure}

Before closing this section, let us establish the inclusion relation of the class $\Scar$ with the classes $\mathcal{S}^*_{SG}:=\mathcal{S}^*(2/(1+e^{-z}))$, $\mathcal{S}^*_{cosh}:=\mathcal{S}^*(\cosh z)$ and $\mathcal{S}^*_{R}:=\mathcal{S}^*(\psi_R)$ where
\begin{equation}\label{eqr}
\psi_R(z)=1+\frac{z}{k}\left(\frac{k+z}{k-z}\right),\quad k=\sqrt{2}+1.
\end{equation}
These classes were introduced by Goel and Kumar \cite{GOEL},  Alotaibi \emph{et al.} \cite{COSH} and Kumar and Ravichandran \cite{SUSHIL} respectively. Figure \ref{subset} shows that the image domains $|\log(w/(2-w))|<1$, $|\log(w+\sqrt{w^2-1})|<1$ and $\psi_R(\mathbb{D})$ of the functions $2/(1+e^{-z})$, $\cosh z$ and $\psi_R$ respectively under the unit disk lie inside $\Omega_{car}$. As a result $\mathcal{S}^*_{SG}$, $\mathcal{S}^*_{cosh}$ and $\mathcal{S}^*_{R}$ are subclasses of $\Scar$. Consequently the $\Scar$-radius of the classes $\mathcal{S}^*_{SG}$, $\mathcal{S}^*_{cosh}$ and $\mathcal{S}^*_{R}$ is unity. The radii constants associated with other subclasses of starlike functions are investigated in the next section.

\section{$\Scar$-Radius}
Recently, several other subclasses of starlike functions have been defined using the unified class introduced by Ma and Minda\cite{MaMinda} based on some interesting curves in the Euclidean plane. The classes $\mathcal{S}^*_{C}:=\mathcal{S}^*(1+4z/3+2z^2/3)$, $\mathcal{S}^*_{lim}:=\mathcal{S}^*(1+\sqrt{2}z+z^2/2)$, $\mathcal{S}^*_{ne}:=\mathcal{S}^*(1+z-z^3/3)$ and $\mathcal{S}^*_{RL}:= \mathcal{S}^*(\sqrt{2}-(\sqrt{2}-1)((1-z)/(1+2(\sqrt{2}-1)z))^{1/2})$ are the classes of starlike functions associated with cardioid $(9u^2+9v^2-18u+5)^2-16(9u^2+9v^2-6u+1)=0$, limacon $(4u^2+4v^2-8u-5)^2+8(4u^2+4v^2-12u-3)=0$, nephroid $((u-1)^2+v^2-4/9)^3-4v^2/3=0$ and left-half of the shifted lemniscate of Bernoulli $|(w-\sqrt{2})^2-1|=1$ respectively. Other classes associated with univalent functions include $\mathcal{S}^*_{\leftmoon}:=\mathcal{S}^*(z+\sqrt{1+z^2})$ and $\mathcal{S}^*_{sin}:=\mathcal{S}^*(1+\sin z)$. These classes have been introduced and studied in \cite{GANDHI, WANI, MEND, KANSH, CHO1, RAINA, PSHARMA, YUNUS, WANI2}. In this section, we will determine $\Scar$-radius for all the classes discussed here and in the previous section. Firstly, let us determine the $\Scar$-radius of the class $ \mathcal{S}^*[A,B]$, where $-1\leq B<A\leq 1$.

\begin{theorem}\label{result3}
Let $-1\leq B<A\leq 1$. If $B\geq 0$, then the $\mathcal{S}^*_{car}$-radius for the class $ \mathcal{S}^*[A,B]$ is $R_2=\min\{1,1/(2A-B)\}$. If $B<0$, then the $\mathcal{S}^*_{car}$-radius for the class $\mathcal{S}^*[A,B]$ is $R_2$ if $R_2\leq R_1$ and $R_3$ if $R_2>R_1$, where $R_1=1/\sqrt{B(3B-2A)}$ and $R_3=\min\{1,3/(2A-5B)\}$.
\end{theorem}
\begin{proof}
Let $f\in\mathcal{S}^*[A,B]$. For $|z|=r$, we have
	\begin{align}\label{eqn1}
	\left|\frac{zf'(z)}{f(z)}-\frac{1-ABr^2}{1-B^2r^2}\right|\leq \frac{(A-B)r}{1-B^2r^2}
	\end{align}
by \cite[Lemma 2.1, p.\ 267]{RAVI}. Set $a(r):=(1-ABr^2)/(1-B^2r^2)$ and $R(r):=(A-B)r/(1-B^2r^2).$ If $B\geq0$, then $a(r)\leq 1$ so that $f$ lies in the domain $\Omega_{car}$ provided $R(r)\leq a(r)-1/2$ by using \lemref{result7}. The last inequality is valid for $r\leq R_2$.

Let us now suppose that $B<0$. Then $3B-2A<2B-2A<0$ so that $B(B-2A)>0$. Similarly, $2A-B>0$ and $2A-5B>0$. Thus $R_1$, $R_2$ and $R_3$ are well-defined real numbers. Also, in this case, $a(r)>1$. If $a(r)\leq 3/2$, then $r\leq R_1$. By \lemref{result7}, it follows that the disk $|w-a(r)|<R(r)$ lies inside $\Omega_{car}$ provided $R(r)\leq a(r)-1/2$. This gives $r\leq R_2$. Therefore if $R_2\leq R_1$, then both the conditions $a(r)\leq 3/2$ and $R(r)\leq a(r)-1/2$ are satisfied simultaneously for $r\leq R_2$.

Assume that $R_2> R_1$. Note that $r\geq R_1$ if and only if $a(r)\geq 3/2.$ In particular, if $r\geq R_2$, then $a(r)\geq 3/2$. Again by applying \lemref{result7}, the disk $|w-a(r)|<R(r)$ lies inside $\Omega_{car}$ if $R(r)\leq 5/2-a(r)$, which yields $r\leq R_3$.

The result is sharp for the function $f_0$ given by
	\[f_0(z)=\begin{cases}
	z(1+Bz)^{\frac{A-B}{B}},& B\neq 0\\
	z \exp(Az),& B=0
	\end{cases}\]
which satisfies $z f'(z)/f(z)=(1+Az)/(1+Bz)$. Note that
\[\left.\frac{zf'(z)}{f(z)}\right|_{z=-R_2}=\frac{1}{2}\quad \mbox{and}\quad \left.\frac{zf'(z)}{f(z)}\right|_{z=R_3}=\frac{5}{2}.\qedhere\]
\end{proof}

Since the class $ \mathcal{S}^*[A,B]$ contains well-known classes of starlike functions for several choices of $A$ and $B$, therefore Theorem \ref{result3} leads to the following corollary.

\begin{corollary}\label{cor5.3}
(i) For $0\leq \alpha<1$, the $\Scar$-radius for the class $\mathcal{S}^*(\alpha)$ is $1/(3-4\alpha)$ if $0\leq\alpha\leq 1/4$ and $3/(7-4\alpha)$ if $1/4< \alpha<1$. (ii) For $0\leq \alpha<1$, the $\Scar$-radius for the class $\mathcal{S}^*[1-\alpha,0]$ is $1/(2(1-\alpha))$ if $0<\alpha< 1/2$ and $1$ if $1/2\leq \alpha<1$. (iii) For $0<\alpha\leq 1$, the $\Scar$-radius for the class $\mathcal{S}^*[\alpha,-\alpha]$ is $1$ if $0<\alpha\leq 1/3$ and $1/(3\alpha)$ if $1/3<\alpha\leq 1$. (iv) For $M>1/2$, the $\Scar$-radius for the class $\mathcal{S}^*[1,-(M-1)/M]$ is $M/(3M-1)$.
\end{corollary}

\begin{proof}
For (i), by taking $A=(1-2\alpha)$ and $B=-1$ in Theorem \ref{result3}, we have $R_1=1/\sqrt{5-4\alpha}$, $R_2=1/(3-4\alpha)$ and $R_3=3/(7-4\alpha)$. The result now follows by applying the case $B<0$. By Corollary \ref{cor4.2}, the $\Scar$-radius for the classes $\mathcal{S}^*[1-\alpha,0]$, $\alpha\in[1/2,1)$ and $\mathcal{S}^*[\alpha,-\alpha]$, $\alpha\in (0,1/3]$ is unity. The case $B=0$ is applicable in part (ii), while in (iii), $1/(3\alpha)=R_2<R_1=1/(\sqrt{5}\alpha)$.

For the last part (iv), if we take $A=1$ and $B=-(M-1)/M$, then it is easy to see that $B\geq 0$ if and only if $M\leq 1$. Therefore if $M\leq 1$ then the $\Scar$-radius is $R_2=M/(3M-1)$. If $M>1$, then $R_1=M/\sqrt{(M-1)(5M-3)}$ and $R_2\leq R_1$, therefore $\Scar$-radius is $R_2$ in this case as well.
\end{proof}

By Corollary \ref{cor5.3}, it is evident that the $\Scar$-radius for the class $\mathcal{S}^*$ is $1/3$ and sharpness holds for the Koebe function $k(z)=z/(1-z)^2.$ In particular, the function $z/(1-Az)^2 \in \Scar$ if and only if $|A|\leq 1/3$. By classical Mark-Strohh\"acker Theorem \cite[Theorem 2.6a, 57]{MILLER}, $\mathcal{K}\subset\mathcal{S}^*(1/2).$ By Corollary \ref{cor5.3}, the $\Scar$-radius for the class $\mathcal{K}$ is at least $3/5.$ In fact, this radius is best possible. To see this, consider the function $l(z)=z/(1-z)\in \mathcal{K}$. This function satisfies $z_0 l'(z_0)/l(z_0)=5/2$ at $z_0=3/5$. Therefore the $\Scar$-radius for the class $\mathcal{K}$ is $3/5$. Now, we will determine the $\Scar$-radius for other subclasses of starlike functions defined earlier in this paper.

\begin{theorem} \label{th5.3}
The $\Scar$-radius for various subclasses of starlike functions is given as follows:
\begin{center}
\begin{tabular}{QQA}
\toprule
\mbox{S.No.} & \mbox{Class} & \mbox{Radius}\\\midrule
(a)&\mathcal{S}^*(\sqrt{1+cz}) & r_1=\displaystyle\frac{3}{4c}\quad \left(\displaystyle\frac{3}{4}<c\leq 1\right) \\\midrule
(b)&\mathcal{S}^*_L(\alpha) & r_2=\displaystyle\frac{3-4\alpha}{4(1-\alpha)^2}\quad \left(0\leq \alpha< \frac{1}{2}\right) \\\midrule
(c)&\mathcal{S}^*_e(\alpha) & r_3=\log\left(\displaystyle\frac{2(1-\alpha)}{1-2\alpha}\right)\quad \left(0\leq \alpha<\displaystyle\frac{e-2}{2(e-1)}\right) \\\midrule
(d)&\mathcal{S}^*_{RL} & r_4=\displaystyle\frac{1}{82}(39+17\sqrt{2})\approx 0.7688 \\\midrule
(e)&\mathcal{S}^*_C & r_5=\displaystyle\frac{1}{2}\approx 0.5 \\\midrule
(f)&\mathcal{S}^*_{lim} & r_6=\sqrt{2}-1\approx 0.414 \\\midrule
(g)&\mathcal{S}^*_{\leftmoon} & r_7=\displaystyle\frac{3}{4}\approx 0.75 \\\midrule
(h)&\mathcal{S}^*_{sin} & r_8=\sin^{-1}\left(\displaystyle\frac{1}{2}\right)\approx 0.523598 \\\midrule
(i)&\mathcal{S}^*_{ne} & r_{9}\approx 0.557875  \\\bottomrule
\end{tabular}
\end{center}
Here $r_9$ is the smallest positive real root of the equation $2r^3-6r+3=0$ in $(0,1)$. All bounds are sharp.
\end{theorem}
\begin{proof}
Let $|z|=r$ and $\phi_{car}(\mathbb{D})=\Omega_{car}$ where $\phi_{car}$ is given by \eqref{eq1.1}.

(a) If $c\in (0,3/4]$, then the $\Scar$-radius of the class $\mathcal{S}^*(\sqrt{1+cz})$ is $1$ by Theorem \ref{th4.1}(vi). For $c\in (3/4,1]$, if $f \in \mathcal{S}^*(\sqrt{1+cz})$, then
\[\left|\frac{zf'(z)}{f(z)}-1\right|\leq 1-\sqrt{1-cr}.\]
Using \lemref{result7}, $ 1-\sqrt{1-cr}\leq 1/2$ which gives $r\leq 3/(4c):=r_1.$ This bound is sharp for the function $f_1$ with $zf'_1(z)/f_1(z)=\sqrt{1+cz}$ and $zf_1'/f_1$ assumes the value $1/2$ at $z=-r_1$.

(b) By Theorem \ref{th4.1}(v), the $\Scar$-radius of the class $\mathcal{S}^*_L(\alpha)$ is $1$ if $\alpha\in [1/2,1)$. Let $f\in\mathcal{S}^*_L(\alpha)$ for $0\leq \alpha< 1/2.$ Then $zf'(z)/f(z)\prec \alpha+(1-\alpha)\sqrt{1+z}$ and
\[\left|\frac{zf'(z)}{f(z)}-1\right| \leq (1-\alpha)(1-\sqrt{1-r}).\]
In view of \lemref{result7}, it is easy to deduce that $(1-\alpha)\left(1-\sqrt{1-r}\right)\leq 1/2$ or equivalently $r\leq (3-4\alpha)/(4(1-\alpha)^2):=r_2.$ This bound is best possible for the function $f_2$ with $zf'_2(z)/f_2(z)=\alpha+(1-\alpha)\sqrt{1+z}$ as $zf'_2/f_2$ takes the value $1/2$ at $z=-r_2$.

(c) For $\alpha \in [(e-2)/(2(e-1)),1)$, the $\Scar$-radius of the class $\mathcal{S}^*_e(\alpha)$ is $1$ by Theorem \ref{th4.1}(iv). If $0\leq \alpha< (e-2)/(2(e-1))$ and $f\in\mathcal{S}^*_e(\alpha)$, then $zf'(z)/f(z)\prec\alpha+(1-\alpha)e^{z}$ which gives
\[\left|\frac{zf'(z)}{f(z)}-1\right|\leq(1-\alpha)(1-e^{-r})\]
so that $(1-\alpha)(1-e^{-r})\leq 1/2$ by Lemma \ref{result7}. Thus $r\leq\log (2(1-\alpha)/(1-2\alpha)):=r_3$. For sharpness, consider the function $f_3$ with $zf_3'(z)/f_3(z)=\alpha+(1-\alpha)e^{z}$. Note that the quantity $zf_3'/f_3$ reduces to $1/2$ at the point $z=-r_3$.

(d) Let $f\in\mathcal{S}^*_{RL}.$ Then
\[\frac{zf'(z)}{f(z)}\prec\sqrt{2}-(\sqrt{2}-1)\sqrt{\frac{1-z}{1+2(\sqrt{2}-1)z}}.\]
A simple computation shows that
\[\left|\frac{zf'(z)}{f(z)}-1\right|\leq 1-\left(\sqrt{2}-(\sqrt{2}-1)\sqrt{\frac{1+r}{1-2(\sqrt{2}-1)r}}\right).\]
By applying \lemref{result7}, we must have
\[1-\left(\sqrt{2}-(\sqrt{2}-1)\sqrt{\frac{1+r}{1-2(\sqrt{2}-1)r}}\right)\leq \frac{1}{2}\]
which yields $r\leq (39+17\sqrt{2})/82:=r_4.$ The result is sharp for the function $f_4$ defined as
\[\frac{zf_4'(z)}{f_4(z)}=\sqrt{2}-(\sqrt{2}-1)\sqrt{\frac{1-z}{1+2(\sqrt{2}-1)z}}\]
which satisfies
\[\left.\frac{zf_4'(z)}{f_4(z)}\right|_{z=-r_4}=\sqrt{2}-\frac{1}{2}(\sqrt{2}-1)\sqrt{11+6\sqrt{2}}.\]
As $(3+\sqrt{2})^2=11+6\sqrt{2}$, therefore it is easy to verify that $zf_4'(z)/f_4(z)$ equals $1/2$ at $z=-r_4$.

\begin{figure}[h]
	\begin{center}
		\subfigure[$r_5$]{\includegraphics[width=1.5in]{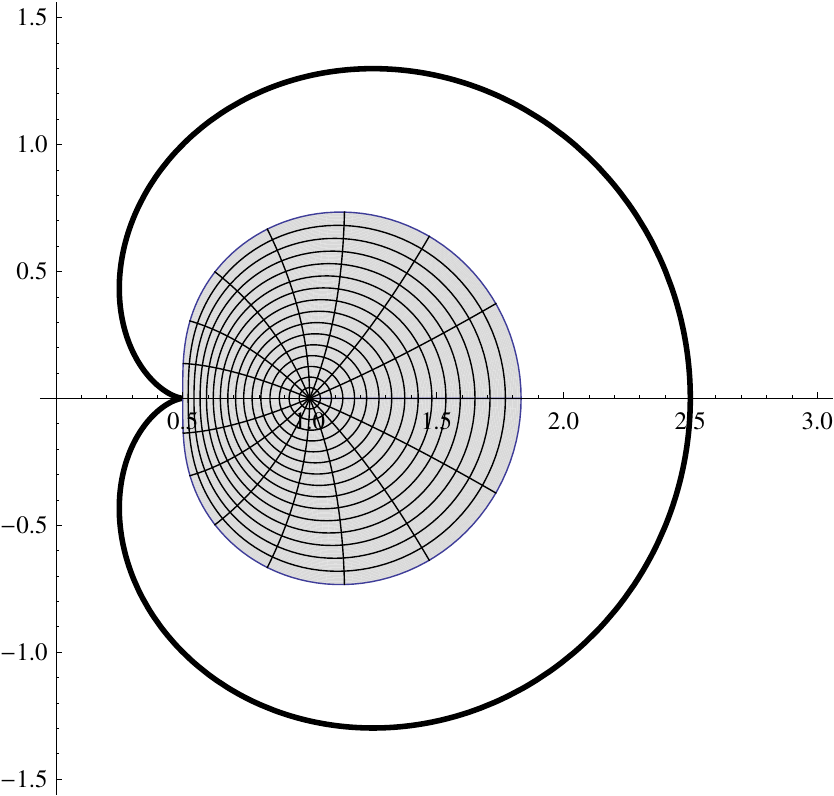}}\hspace{10pt}
		\subfigure[$r_6$]{\includegraphics[width=1.5in]{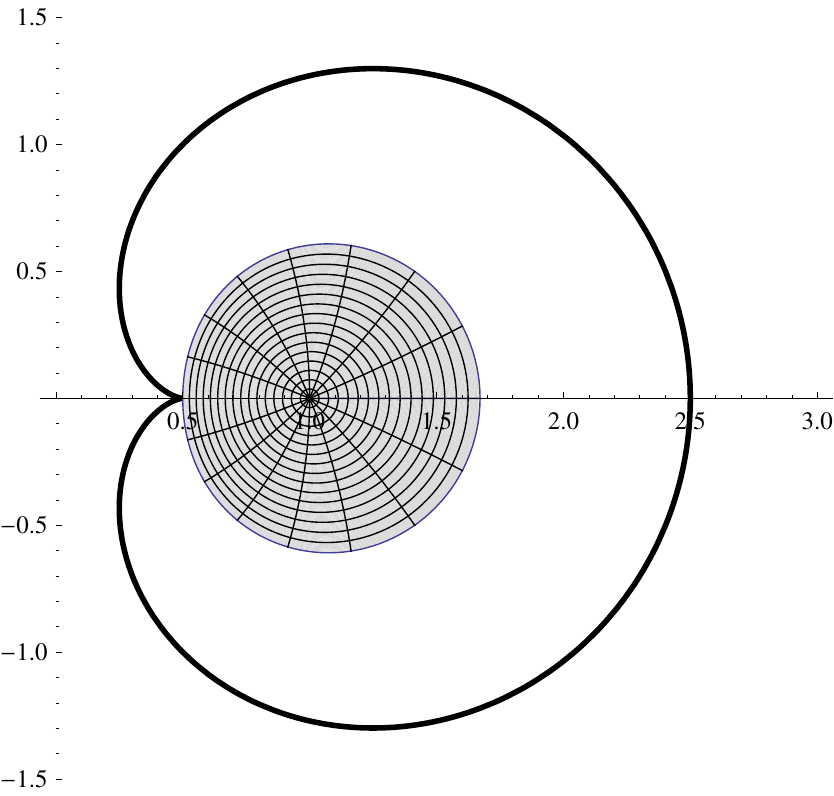}}\hspace{10pt}
\subfigure[$r_7$]{\includegraphics[width=1.5in]{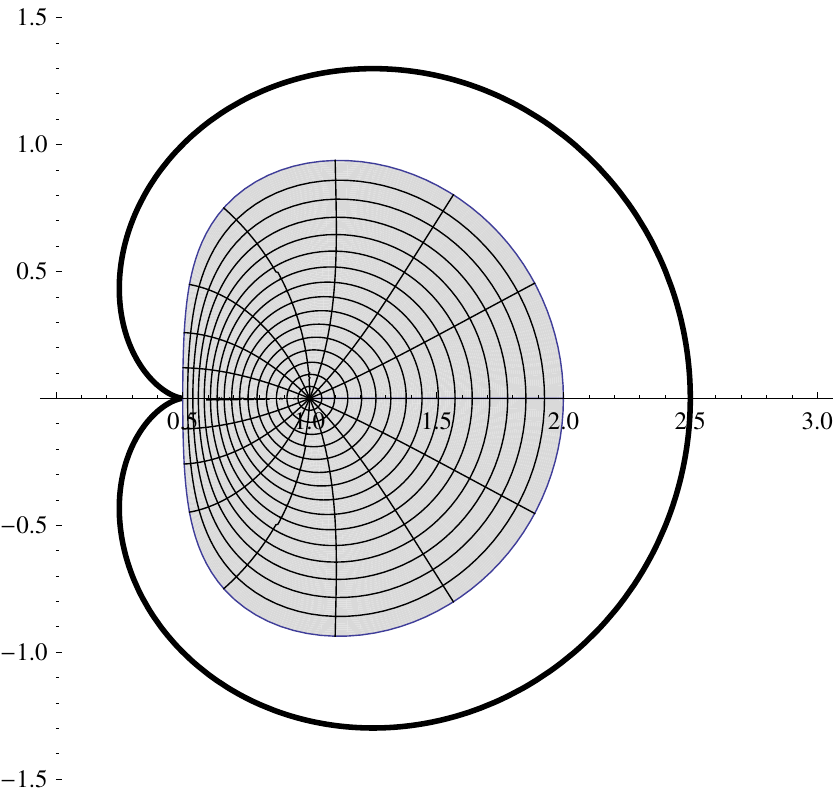}}\hspace{10pt}
\subfigure[$r_8$]{\includegraphics[width=1.5in]{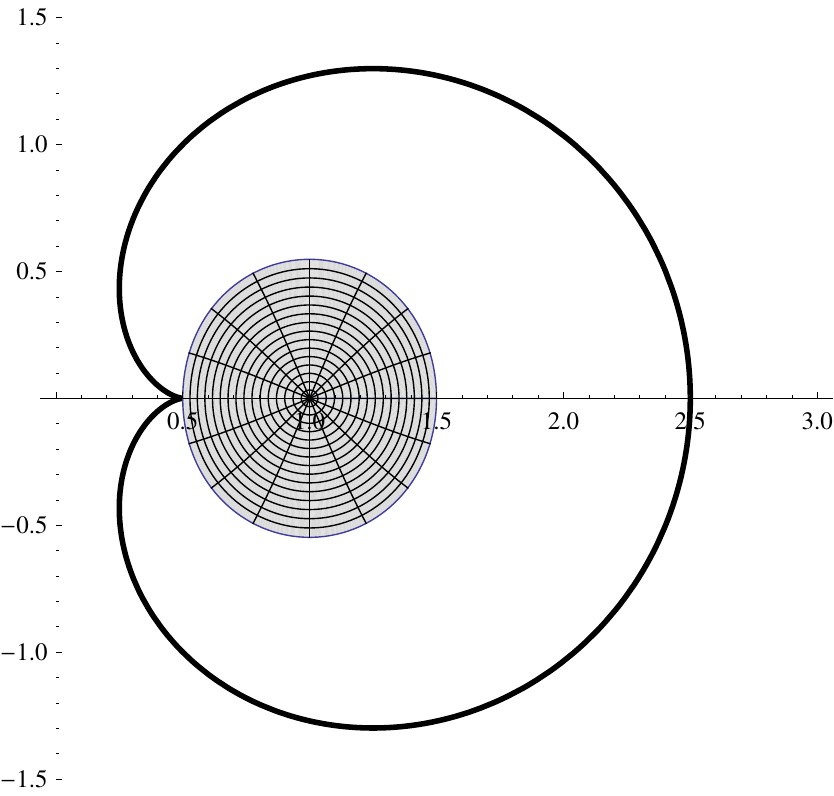}}\hspace{10pt}
\subfigure[$r_{9}$]{\includegraphics[width=1.5in]{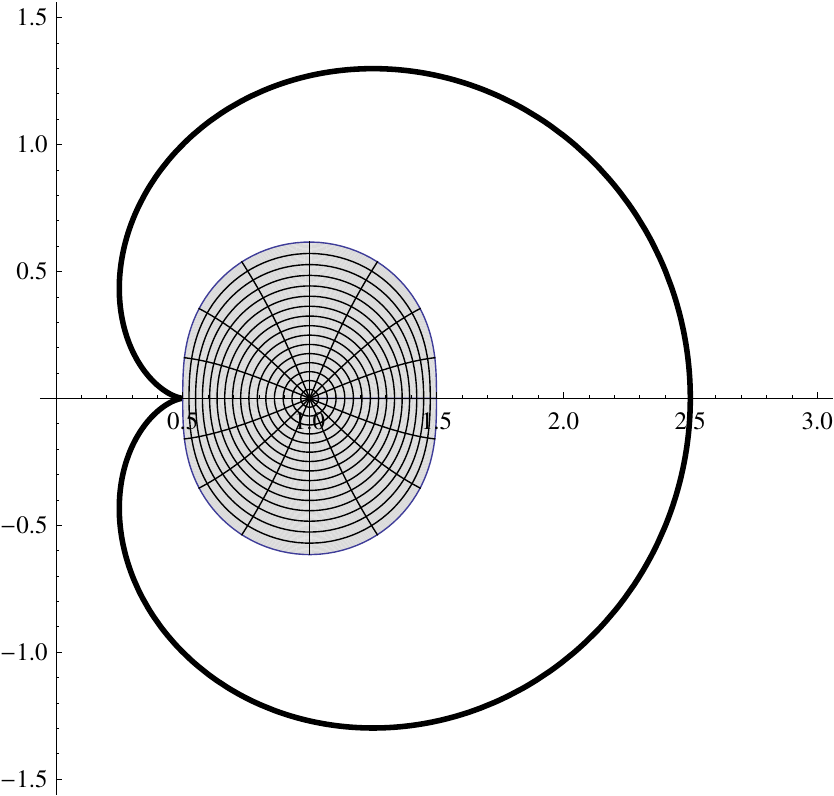}}\hspace{10pt}
		\caption{Image Domains lying inside $\Omega_{car}$.}\label{radius}
	\end{center}
\end{figure}

(e) If $f\in\mathcal{S}^*_C$, then $zf'(z)/f(z)\prec 1+4z/3+2z^2/3:=\psi_C(z)$. We need to find the value of $r$ such that $zf'(z)/f(z)\prec \phi_{car}(z)$ in $\mathbb{D}_r$. It suffices to determine the value of $r$ that satisfies $\psi_C(z)\prec \phi_{car}(z)$ in $\mathbb{D}_r$. For this subordination to happen, it is necessary that
\[\frac{1}{2}=\phi_{car}(-1)\leq \phi_{car}(-r)\leq \psi_C(-r)=1-\frac{4r}{3}+\frac{2r^2}{3}.\]
This gives $r\leq 1/2:=r_5$. In fact, the image of the disk $\mathbb{D}_{r_5}$ under the function $\psi_C$ lies inside the domain $\Omega_{car}$ (see Figure \ref{radius}(a)). Thus the $\Scar$-radius of the class $\mathcal{S}^*_C$ is at least $r_5$. In fact, this bound cannot be further improved as seen by the function $f_5(z)=z \exp((4z+z^2)/3)$. The function $f_5$ satisfies $zf_5'(z)/f_5(z)=1/2$ at $z=-r_5$.

(f) If $f\in\mathcal{S}^*_{lim},$ then $zf'(z)/f(z)\prec 1+\sqrt{2}z+z^2/2:=\psi_{lim}(z).$ A necessary condition for the subordination  $\varphi_{lim}\prec\phi_{car}$ to hold in $\disc_r$ is given by
\[\frac{1}{2}=\phi_{car}(-1)\leq \phi_{car}(-r)\leq\psi_{lim}(-r)=1-\sqrt{2}r+\frac{r^2}{2}\]
which yields $r\leq \sqrt{2}-1:=r_6.$ Figure \ref{radius}(b) shows that $\psi_{lim}(\mathbb{D}_{r_6})\subset \Omega_{car}$. The bound $r_6$ is sharp and is attained by the function $f_6(z)=z \exp(\sqrt{2}z+z^2/4)$. Note that $zf_6'(z)/f_6(z)=1/2$ at $z=-r_6$.

(g) For a function $f\in\mathcal{S}^*_{\leftmoon}$, $zf'(z)/f(z)\prec z+\sqrt{z^2+1}:=\psi_{\leftmoon}.$ In order to compute the $\Scar$-radius for the class $\mathcal{S}^*_{\leftmoon},$ we shall find the value of $r$ such that the subordination $\psi_{\leftmoon}\prec\phi_{car}$ holds in $\disc_r.$ For this to hold, it is necessary that
	\[\frac{1}{2}\leq \phi_{car}(-1)\leq \phi_{car}(-r)\leq \psi_{\leftmoon}(-r)=-r+\sqrt{1+r^2},\]
	which reduces to $r\leq 3/4:=r_7.$ The image domain $\psi_{\leftmoon}(\mathbb{D}_{r_7})$ is contained in $\Omega_{car}$ (see Figure \ref{radius}(c)). This shows that $\Scar$-radius for the class $\mathcal{S}^*_{\leftmoon}$ is at least $r_7$. If we consider the function $f_7$ defined as $zf_7'(z)/f_7(z)=\psi_{\leftmoon}(z)$, then $\psi_{\leftmoon}(-r_7)=1/2$ and this shows that the bound $r_7$ is sharp.

(h) If $f\in\mathcal{S}^*_{sin},$ then $zf'(z)/f(z)\prec 1+\sin z:=\psi_s(z). $ Proceeding in a similar manner, the inequalities
	\[\frac{1}{2}\leq \phi_{car}(-1)\leq \phi_{car}(-r)\leq \psi_s(-r)=1-\sin r\]
give $r\leq \sin^{-1}(1/2):=r_8.$ Also, Figure \ref{radius}(d) depicts that $\Omega_{car}$ contains the image domain $\psi_s(\mathbb{D}_{r_8})$. Thus the $\Scar$-radius for the class $\mathcal{S}^*_{sin}$ is at least $r_8$. For the function
	 \[f_8(z)=z \exp\left(\int_{0}^{z}\frac{\sin t}{t}dt\right)=z+z^2+\frac{z^3}{2}+\frac{z^4}{9}+\cdots,\]
it is easily seen that $zf_8'(z)/f_8(z)=1/2$ at $z=-r_8$.

(i) Let $f\in\mathcal{S}^*_{ne}.$ Then $zf'(z)/f(z)\prec 1+z-z^3/3:=\psi_{ne}(z).$ A similar analysis shows that a necessary condition for the subordination $\psi_{ne}\prec \phi_{car}$ to hold in $\mathbb{D}_r$ is
\[\frac{1}{2}\leq \phi_{car}(-1)\leq \phi_{car}(-r)\leq \psi_{ne}(-r)=1-r+\frac{r^3}{3}.\]
Let $r_{9}$ be the real root of the equation $2r^3-6r+3=0$ in $(0,1)$. The inclusion relation $\psi_{ne}(\mathbb{D}_{r_{9}})\subset \Omega_{car}$ is clearly illustrated in Figure \ref{radius}(e). If a function $f_{9}$ is defined as $zf_{9}'(z)/f_{9}(z)=\psi_{ne}(z)$, then $\psi_{ne}(-r_{9})=1/2$ and this shows that the bound $r_{9}$ is best possible.
\end{proof}

By Theorem \ref{th5.3}, the $\Scar$-radius of the classes $\mathcal{S}^*_L$ and $\mathcal{S}^*_e$ are $3/4$ and $\log2$ respectively. Now, we will compute the $\Scar$-radius for the class $\mathcal{S}$. If $f\in \mathcal{S}$, then $f$ is starlike in $|z|<\tanh(\pi/4)\approx 0.655794$ by \cite[p.\ 98]{DUREN}. Consequently,
\[\frac{zf'(z)}{f(z)}\prec \frac{1+z}{1-z}:=p(z)\quad \mbox{in }\mathbb{D}_{\tanh(\pi/4)}.\]
It suffices to determine the value of $r$ that satisfies $p(z)\prec \phi_{car}(z)$ in $\mathbb{D}_r$. A necessary condition for this subordination to hold is
\[\frac{1}{2}=\phi_{car}(-1)\leq \phi_{car}(-r)\leq p(-r)=\frac{1-r}{1+r}.\]
This gives $r\leq 1/3$. Also, $p(\mathbb{D}_{1/3})\subset \Omega_{car}$ (see Figure \ref{univalent}). Therefore it follows that $p(z)\prec \phi_{car}(z)$ in $\mathbb{D}_{1/3}$. Hence $zf'(z)/f(z)\prec \phi_{car}(z)$ in $|z|<\min\{\tanh(\pi/4),1/3\}=1/3$. The result is sharp for the Koebe function $k(z)=z/(1-z)^2$. This also shows that the $\Scar$-radius for the class $\mathcal{C}$ of close-to-convex functions is $1/3$.

\begin{figure}[h]
	\begin{center}
		\includegraphics[width=2in]{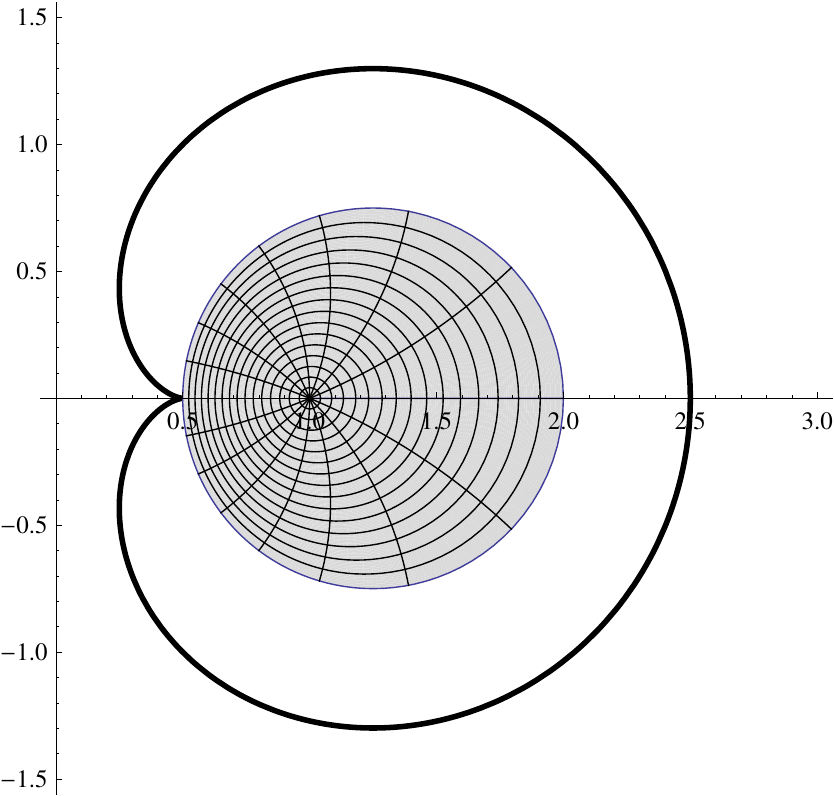}
		\caption{Image of subdisk $\mathbb{D}_{1/3}$ under $p(z)=(1+z)/(1-z)$.}\label{univalent}
	\end{center}
\end{figure}

Till now, the $\Scar$-radius of subclasses for univalent functions was investigated. The next result evaluates the $\Scar$-radius for the classes $\mathcal{BL}(\alpha)$ and $\mathcal{M}(\beta)$ satisfying $zf'(z)/f(z)\prec 1+z/(1-\alpha z^2)$ ($0\leq \alpha <1$) and $\RE (zf'(z)/f(z))<\beta$ ($\beta>1$) respectively which contain non-univalent functions as well. These classes were introduced by Kargar \emph{et al.} \cite{KARGAR} and Uralegaddi \emph{et al.} \cite{MBETA}.

\begin{theorem}
(i) The $\Scar$-radius of the class $\mathcal{BL}(\alpha)$, $0\leq \alpha <1$ is $1/(1+\sqrt{1+\alpha})$. (ii) The $\Scar$-radius of the class $\mathcal{M}(\beta)$, $\beta>1$ is $1/(4\beta-3).$

\end{theorem}

\begin{proof}
For the proof of (i), let $f\in\mathcal{BL}(\alpha)$. Then $zf'(z)/f(z)\prec 1+z/(1-\alpha z^2)$ so that
\[\left|\frac{zf'(z)}{f(z)}-1\right|\leq \frac{r}{1-\alpha r^2}.\]
By \lemref{result7}, the above disk lies in the domain $\Omega_{car}$ if $r/(1-\alpha r^2)\leq 1/2$ which is true provided $r\leq 1/(1+\sqrt{1+\alpha}):=r_1.$ The result is sharp for the function
\[f_1(z)=z\left(\frac{1+\sqrt{\alpha}z}{1-\sqrt{\alpha}z}\right)^{1/(2\sqrt{\alpha})}\]
which satisfies $zf'_1(z)/f_1(z)=1+z/(1-\alpha z^2)$ and $zf_1'/f_1$ assumes the value $1/2$ at $z=-r_1$.

In order to prove (ii), if $f\in\mathcal{M}(\beta)$ then
	\[\left|\frac{zf'(z)}{f(z)}-\frac{1+(1-2\beta)r^{2}}{1-r^{2}}\right|\leq\frac{2r(\beta-1)}{1-r^{2}}.\]
As $\beta>1$, $(1+(1-2\beta)r^{2})/(1-r^{2})< 1$ so that \lemref{result7} implies
	\[\frac{2r(\beta-1)}{1-r^{2}}\leq\frac{1+(1-2\beta)r^{2}}{1-r^{2}}-\frac{1}{2}.\]
This gives $r\leq 1/(4\beta-3)$. This bound is sharp as the function $f_0(z)=z(1-z)^{2(\beta-1)} \in \mathcal{M}(\beta)$ satisfies $zf_0'(z)/f_0(z)=1/2$ at $z=1/(4\beta-3)$.
\end{proof}

\section{Radii Constants for the class $\mathcal{S}^*_{car}$}
In this section, we will determine various radii constants for the class $\mathcal{S}^*_{car}$ by making use of the corresponding result of Lemma \ref{result7} for the other subclasses of starlike functions and the geometrical considerations for a subordination relation to hold in a subdisk $|z|<r$.

\begin{theorem}
For the class $\Scar$, the following radii constants are sharp:
\begin{enumerate}
		\item [(a)] For $0\leq \alpha<1$, the $\mathcal{S}^*(\alpha)$-radius is $s_1:=s_1(\alpha)$ where
		\[s_1(\alpha)=\begin{cases}
		1,& 0\leq \alpha\leq 1/4,\\
		\sqrt{\displaystyle\frac{3-4\alpha}{2}},& 1/4<\alpha\leq 5/8,\\
		1-\sqrt{2\alpha-1},& 5/8<\alpha<1.
		\end{cases}\]
		\item[(b)] For $0\leq \alpha<1$, the $\mathcal{S}^*_L(\alpha)$-radius is $s_2:=s_2(\alpha)=-1+((2\sqrt{2}-1)-2(\sqrt{2}-1)\alpha))^{1/2}$. In particular, $s_2(0)=-1+(2\sqrt{2}-1)^{1/2}$.
		\item[(c)] The $\mathcal{S}^*_{RL}$-radius is $s_3:=-1+(1+2(-\gamma+\sqrt{\gamma})^{1/2})^{1/2}\approx 0.253734$,
where $\gamma=2\sqrt{2}-2$.
		\item[(d)] The $\mathcal{S}^*_{R}$-radius is $s_4:=1-(4\sqrt{2}-5)^{1/2}\approx 0.189535.$
\item[(e)] The $\mathcal{S}^*_{sin}$-radius is $s_5:=-1+\sqrt{1+2\sin(1)}\approx 0.637969.$
\item[(f)] The $\mathcal{S}^*_{cosh}$-radius is $s_6:=-1+\sqrt{-1+2 \cosh(1)}\approx 0.444355.$
\item[(g)] The $\mathcal{S}^*_{ne}$-radius is $s_7:=(\sqrt{21}-3)/3\approx 0.527525.$
\item[(h)] The $\mathcal{S}^*_{SG}$-radius is  \[s_8:=-1+\sqrt{1+\frac{2(e-1)}{e+1}}\approx 0.387168\]
\item[(i)] For $0\leq \alpha<1$, the $\mathcal{S}^*[1-\alpha,0]$-radius is $s_9:=-1+\sqrt{3-2\alpha}$.
\item[(j)] For $0< \alpha\leq 1$, the $\mathcal{S}^*[\alpha,-\alpha]$-radius is $s_{10}:=s_{10}(\alpha)$ where
\[s_{10}(\alpha):=\begin{cases}
		w_\alpha,& 0<\alpha\leq \alpha^*,\\
		1,& \alpha^*\leq \alpha\leq 1.
		\end{cases}\]
where
\[w_\alpha=\frac{2\alpha}{\sqrt{1-\alpha^2}}\sqrt{\frac{2}{\sqrt{1+3\alpha^2}}-1}\]
and $\alpha^*=((5+2 \sqrt{13})/27)^{1/2}\approx 0.672505$.
\item[(k)] For $M>1/2$, the $\mathcal{S}^*[1,-(M-1)/M]$-radius is $s_{11}:=s_{11}(M)$ where
		\[s_{11}(M)=\begin{cases}
		a_M,& 1/2\leq M\leq M^*,\\
		b_M,& M^*\leq M< (3+\sqrt{5})/4,\\
1,& M\geq (3+\sqrt{5})/4,
		\end{cases}\]
where $a_M=-1+\sqrt{M-1}$,
\[b_M=\sqrt{2\sqrt{2}M\sqrt{\frac{M-1}{2M-1}}-2(M-1)}\]
and $M^*\approx 1.1423$ is the root of the equation $a_M=b_M$.
\item[(l)] The $\mathcal{S}^*_{C}$-radius is $s_{12}:=1$.
\item[(m)] For $\beta>1$, $\mathcal{M}(\beta)$-radius is $s_{13}:=s_{13}(\beta)$ where
		\[s_{13}(\beta)=\begin{cases}
		\sqrt{2\beta-1}-1,& 1<\beta\leq 5/2,\\
		1,&\beta\geq 5/2.
		\end{cases}\]
\end{enumerate}
\end{theorem}

\begin{proof}
Since $f \in \Scar$, $zf'(z)/f(z)\prec \phi_{car}(z)$ where $\phi_{car}$ is given by \eqref{eq1.1} and $\phi_{car}(\mathbb{D})=\Omega_{car}$. Let $|z|=r$. Let $f_{car}$ denotes the function given by \eqref{fcar}.
	
For $(a)$, by Theorem \ref{th4.1}(i), $f \in \mathcal{S}^*(\alpha)$ for $0\leq \alpha\leq 1/4$. Let $1/4<\alpha<1$. For $0<r\leq 1/2$, Lemma \ref{result1} gives \[\RE\left(\frac{zf'(z)}{f(z)}\right)>\min_{|z|=r} \RE(\phi_{car}(z))=1-r+\frac{r^2}{2}>\alpha\]
if $r<1-\sqrt{2\alpha-1},$ where $5/8<\alpha<1.$ Similarly for the case when $1/2\leq r<1,$ it is easily seen that $\RE(zf'(z)/f(z))>(3-2r^2)/4>\alpha$ provided $r<\sqrt{(3-4\alpha)/2}$, where $1/4<\alpha \leq 5/8.$ The result is sharp for the function $f_{car}$.

 For $f\in\Scar,$ $zf'(z)/f(z)\prec 1+z+z^2/2$ so that
 \begin{equation}\label{eq6.1}
 \left|\frac{zf'(z)}{f(z)}-1\right|\leq r+\frac{r^2}{2}.
 \end{equation}
 For $(b)$, the disk \eqref{eq6.1} lies inside the domain $|((w-\alpha)/(1-\alpha))^2-1|<1$ provided $r+r^2/2\leq (\sqrt{2}-1)(1-\alpha)$ by \cite[Lemma 2.3, p.\ 238]{KHATTER}. This implies $r\leq -1+((2\sqrt{2}-1)-2(\sqrt{2}-1)\alpha))^{1/2}:=s_2(\alpha)$. This bound is best possible for the function $f_{car}$ since $zf'_{car}(z)/f_{car}(z)=\alpha+(1-\alpha)\sqrt{2}$ at $z=s_2$. In $(c)$, the disk \eqref{eq6.1} lies inside the domain $|(w-\sqrt{2})^2-1|<1$ if $r+r^2/2\leq ((2\sqrt{2}-2)^{1/2}-(2\sqrt{2}-2))^{1/2}$ by \cite[Lemma 3.2]{MEND}.  This simplifies to $r\leq s_3\approx 0.2537371.$ The sharpness of the radii $s_2(0)$ and $s_3$ is depicted in Figure \ref{abc}.

 \begin{figure}[h]
	\begin{center}
		\subfigure[$s_2(0)$]{\includegraphics[width=1.2in]{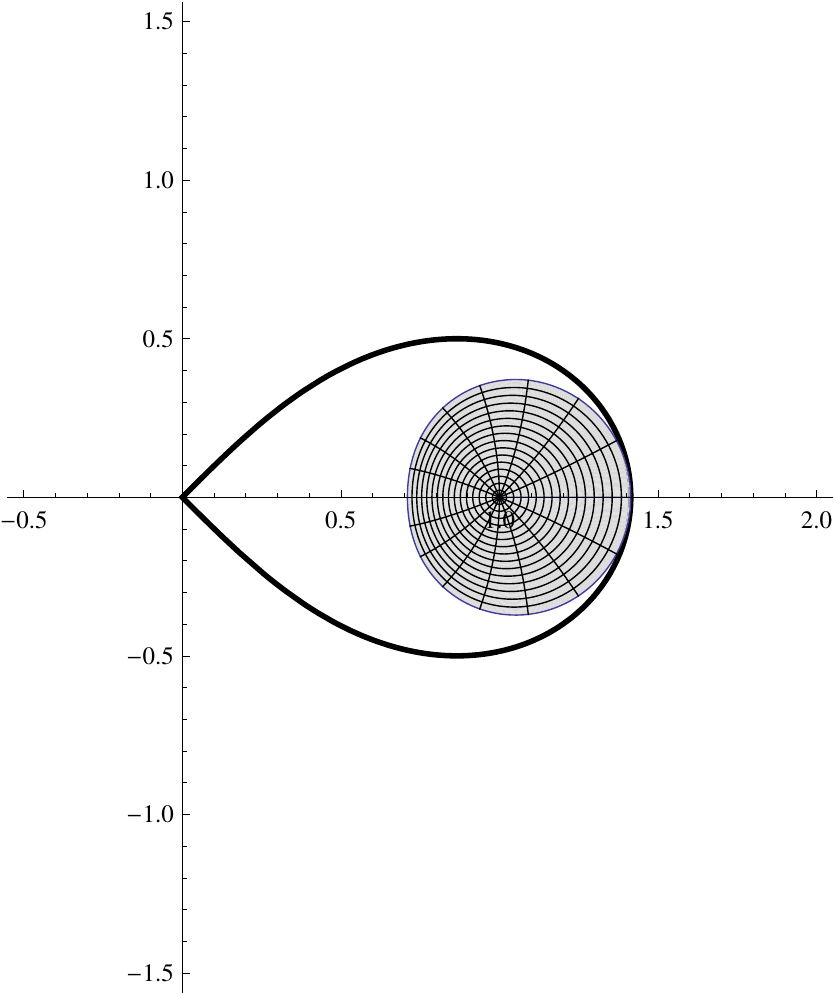}}\hspace{10pt}
		\subfigure[$s_3$]{\includegraphics[width=1.2in]{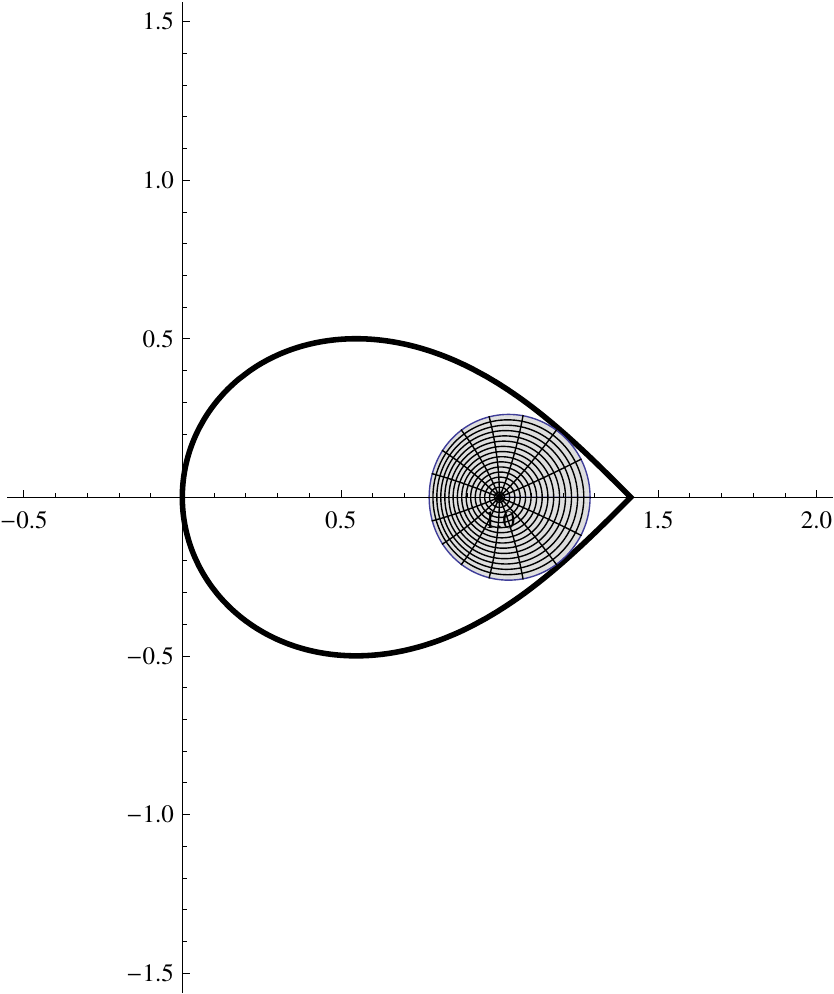}}\hspace{10pt}
\subfigure[$s_7$]{\includegraphics[width=1.2in]{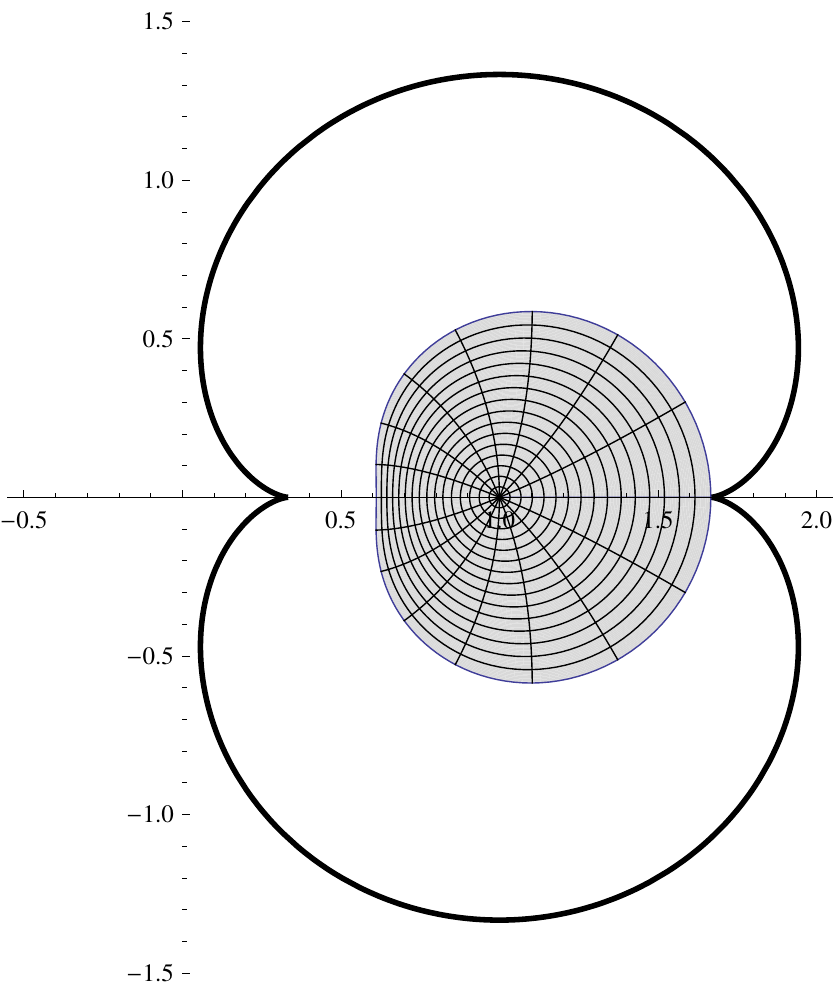}}\hspace{10pt}
\subfigure[$s_8$]{\includegraphics[width=1.2in]{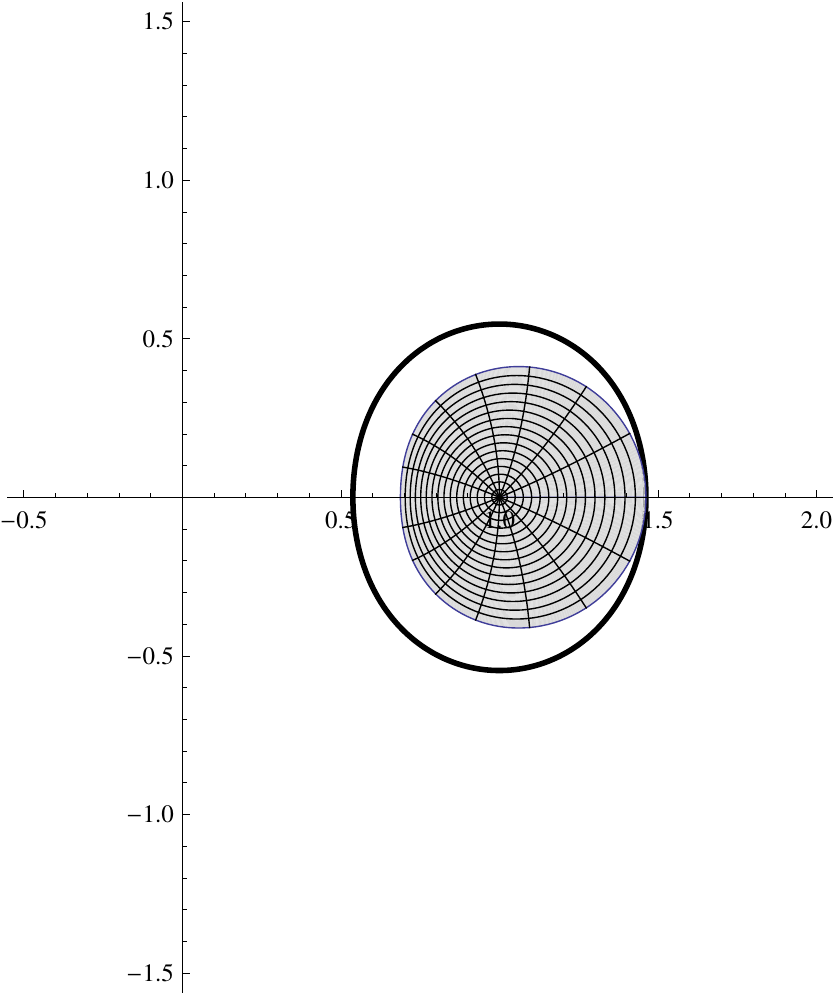}}
		\caption{Sharpness of $\mathcal{S}^*_{L}$, $\mathcal{S}^*_{RL}$, $\mathcal{S}^*_{ne}$ and $\mathcal{S}^*_{SG}$ radii for the class  $\Scar$.}\label{abc}
	\end{center}
\end{figure}

For proving (d), observe that a necessary condition for the subordination $\phi_{car}(z)\prec \psi_R(z)$ to hold in $\mathbb{D}_r$ is
\[2(\sqrt{2}-1)\leq \psi_{R}(-1)\leq \psi_{R}(-r)\leq \phi_{car}(-r)=1-r+\frac{r^2}{2}\]
where $\psi_R$ is given by \eqref{eqr}. This is possible if $r\leq 1-(4\sqrt{2}-5)^{1/2}:=s_4.$ In fact, $\phi_{car}(\mathbb{D}_{s_4})\subset \psi_R(\mathbb{D})$ (see Figure \ref{def}(a)). Thus the $\mathcal{S}^*_{R}$-radius of the class $\Scar$ is at least $s_4$. This bound cannot be further improved as seen by considering the function $f_{car}$. Similarly, for part (e), the relation
\[1+r+\frac{r^2}{2}=\phi_{car}(r)\leq\psi_s(r)\leq \psi_s(1)=1+\sin (1)\]
is necessary for the subordination $\phi_{car}(z)\prec \psi_s(z)$ to satisfy in $\mathbb{D}_r$, where $\psi_s$ is defined in Theorem \ref{th5.3}(h). This gives $r\leq -1+\sqrt{1+2\sin(1)}:=s_5.$ Also, $\phi_{car}(\mathbb{D}_{s_5})\subset \psi_s(\mathbb{D})$ by Figure \ref{def}(b) so that $\mathcal{S}^*_{sin}$-radius is at least $s_5$. The bound is sharp for the function $f_{car}$. The same procedure can be applied in part (f) to show that $\mathcal{S}^*_{cosh}$-radius is $s_6:=-1+\sqrt{-1+2 \cosh(1)}$ by considering the inequality $1+r+r^2/2\leq \cosh (1)$ and Figure \ref{def}(c) illustrating that the image domain of the function $\cosh z$ under the unit disk contains $\phi_{car}(\mathbb{D}_{s_6})$.

\begin{figure}[h]
	\begin{center}
		\subfigure[$s_4$]{\includegraphics[width=1.5in]{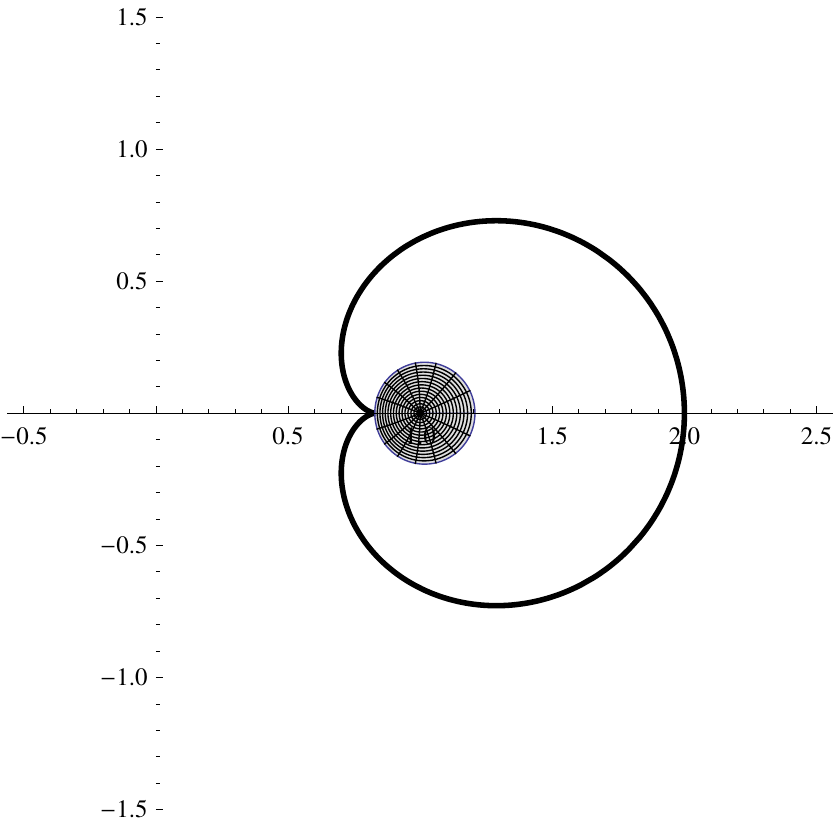}}\hspace{10pt}
		\subfigure[$s_5$]{\includegraphics[width=1.5in]{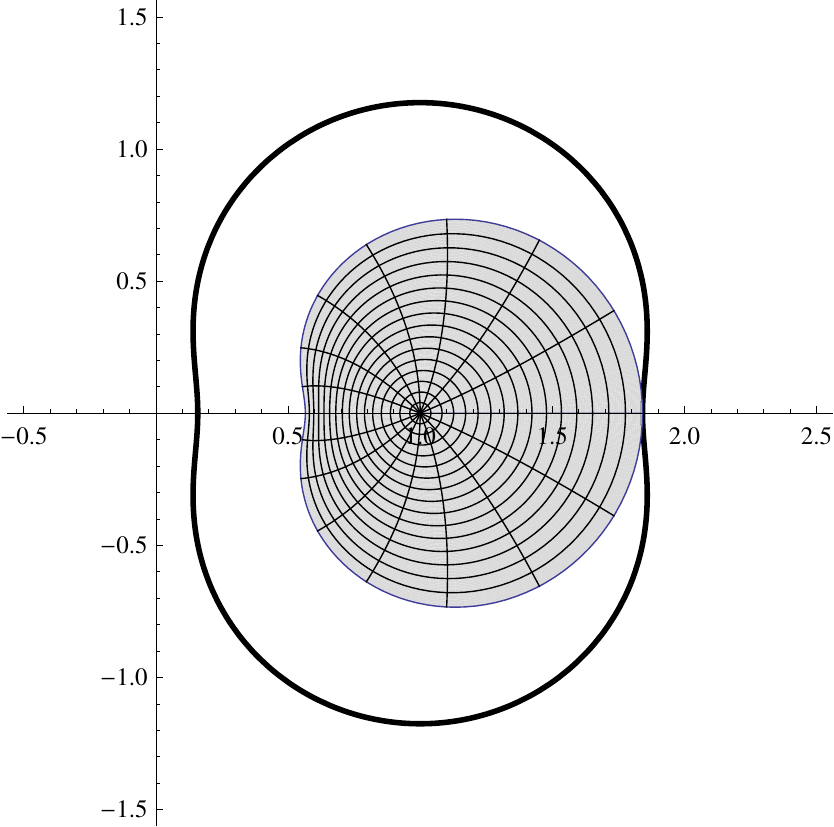}}\hspace{10pt}
\subfigure[$s_6$]{\includegraphics[width=1.5in]{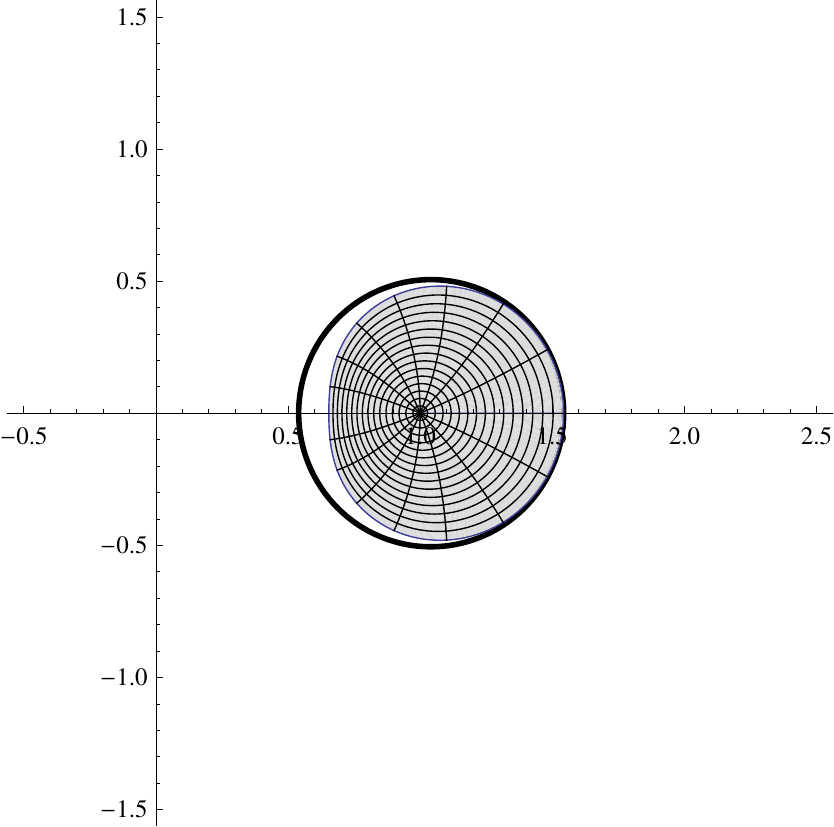}}
		\caption{Inclusions associated with a rational, trigonometric and hyperbolic function.}\label{def}
	\end{center}
\end{figure}

In order to prove (g), note that the disk \eqref{eq6.1} lies inside the domain $\psi_{ne}(\mathbb{D})$ where $\psi_{ne}$ is defined in Theorem \ref{th5.3}(i) if $r+r^2/2\leq 2/3$ (see \cite{WANI}). This simplifies to $r\leq(\sqrt{21}-3)/3:=s_7.$ The result is sharp for the function $f_{car}(z)=z \exp(z+z^2/4)$ as $zf_{car}'(z)/f_{car}(z)=5/3=\psi_{ne}(1)$ at $z=s_7$. In the similar fashion, the part (h) can proved by applying \cite[Lemma 2.2]{GOEL} to show that the disk \eqref{eq6.1} lies in the domain $\psi_{SG}(\mathbb{D})$ if $r+r^2/2\leq (e-1)/(e+1)$, where $\psi_{SH}(z)=2/(1+e^{-z})$. The last inequality gives the desired bound $s_8:=-1+(1+2((e-1)/(e+1)))^{1/2}$ and $zf_{car}'(z)/f_{car}(z)$ equals $2e/(1+e)=\psi_{SH}(1)$ at $z=s_8$. The sharpness of constants $s_7$ and $s_8$ is also depicted graphically in Figure \ref{abc}.

For (i), the disk \eqref{eq6.1} lies inside the domain $|w-1|<1-\alpha$ provided $r+r^2/2\leq 1-\alpha$ which gives $r\leq -1+\sqrt{3-2\alpha}:=s_9$ and at the point $z=s_9$, $zf_{car}'(z)/f_{car}(z)=2-\alpha$. In part (j), we will show that $\phi_{car}(z)\in \{w:|(w-1)/(w+1)|<\alpha\}$ for all $z\in \mathbb{D}_{s_{10}}$. For $z=re^{it}$, we have
\[|\phi_{car}(z)-1|^2=r^2+r^3\cos t+\frac{r^4}{4}\]
and
\[|\phi_{car}(z)+1|^2=4+4r\cos t+r^2+2r^2 \cos(2t)+r^3 \cos t+\frac{r^4}{4}.\]
If we set $x=\cos t$, then the problem now reduces to show that the function
\begin{align*}
g(x,r):&=\alpha^2|\phi_{car}(z)+1|^2-|\phi_{car}(z)-1|^2\\
&=4\alpha^2+4\alpha^2 r x-\alpha^2 r^2+4\alpha^2 r^2 x^2+\alpha^2 r^3 x+\frac{\alpha^2 r^4}{4}-r^2-r^3 x-\frac{r^4}{4}
\end{align*}
is positive for all $x\in [-1,1]$ and $0<r\leq s_{10}$. Note that
\begin{equation}\label{eq6.2}
h(x,r):=\frac{\partial}{\partial x}g(x,r)=-r^3 + 4 r \alpha^2 + r^3 \alpha^2 + 8 r^2 x \alpha^2\quad \mbox{and}\quad \frac{\partial^2}{\partial x^2}g(x,r)=8r^2\alpha^2.
\end{equation}
Consider the following four observations:
\begin{enumerate}
\item $h(x_0,r)=0$ where $x_0=(r^2(1-\alpha^2)-4\alpha^2)/(8\alpha^2 r)$. By \eqref{eq6.2}, $g(x,r)$ has a relative minimum at $x_0$.
\item The quantity
\[g(-1,r)=\frac{1}{4} (2 r(1-\alpha)-(1-\alpha)r^2+4 \alpha) ((1+\alpha)r^2+4 \alpha -2 r (1+\alpha))\]
 is positive if $r\leq 1-\sqrt{(1-3 \alpha)/(1+\alpha)}:=u_\alpha$, $0<\alpha\leq 1/3$ and $g(-1,r)>0$ for all $r$ if $\alpha\geq 1/3$.
 \item The quantity
 \[g(1,r)=\frac{1}{4} (4 \alpha-2 r(1-\alpha)-(1-\alpha)r^2) (4 \alpha+2 r(1+\alpha)+(1+\alpha)r^2)\]
is positive if $r\leq\sqrt{(1+3\alpha)/(1-\alpha)}-1:= v_\alpha$, $0<\alpha\leq 3/7$ and $g(1,r)>0$ for every $r$ if $\alpha\geq 3/7$.
\item $g(x_0,r)>0$ provided $r\leq w_\alpha$, $0<\alpha\leq \alpha^*$ and $g(x_0,r)>0$ for all $r$ and $\alpha\geq \alpha^*$, where $w_\alpha$ and $\alpha^*$ are given in the statement of the theorem. Note that $w_\alpha \leq u_\alpha$ for all $\alpha \in (0,1/3]$ and $w_\alpha \leq v_\alpha$ for all $\alpha \in (0,3/7]$.
 \end{enumerate}
These observations lead us to the desired result with sharpness depicted by the function $f_{car}$.

For proving (k), we need to prove that $\phi_{car}\in \{w:|w-M|<M\}$ for all $|z|<r_{11}$. Let $z=r e^{it}$, $x        =\cos t$. It suffices to show that the function
\begin{align*}
p(x,r)&=M^2- |\phi_{car}(z)-M|^2\\
&=2M-1-Mr^2-r^3 x -\frac{r^4}{4}-2(1-M)rx-2(1-M)r^2x^2
\end{align*}
is positive for all $x\in [-1,1]$ and $0<r\leq s_{11}$. Note that $p(-1,r)>0$ for all $r$, $p(1,r)>0$ for all $r\in (0,a_M]$ with $1/2<M\leq 5/4$ and $p(1,r)>0$ for all $r\in (0,1]$ if $M>5/4$. Also, we have
\begin{equation}\label{eq6.3}
\frac{\partial}{\partial x}p(x,r)=2(M-1)r-r^3+4(M-1)r^2 x\quad \mbox{and}\quad \frac{\partial^2}{\partial x^2}p(x,r)=4(M-1)r^2.
\end{equation}
If $1/2<M\leq 1$, then \eqref{eq6.3} shows that $p(x,r)$ is a decreasing function of $x$ so that $p(x,r)\geq p(1,r)>0$ as $r\leq a_M$. If $M>1$, then $(\partial/\partial x)p(x,r)$ vanishes at $x_0=(r^2-2(M-1))/(4(M-1)r)$ which is a point of relative minima and
\[p(x_0,r)=\frac{4M^2(3-2r^2)+(r^2-2)^2-2M(r^2-2)(r^2-4)}{8(M-1)}>0\]
provided $0<r<b_M$ and $1<M<(3+\sqrt{5})/4$. Let $M^*$ be the root of the equation $a_M=b_M$.  Similar calculation carried out in \cite[Theorem 3.1]{KANSH} shows that if $1<M\leq M^*$, then $p(x,r)>0$ for $r\leq a_M$ and if $M^*\leq M<(3+\sqrt{5})/4$, then $p(x,r)>0$ for $r\leq b_M$. If $M\geq (3+\sqrt{5})/4$, then the result follows by Theorem \ref{th4.1}(viii). Here $a_M$ and $b_M$ are defined in the statement of the theorem.

Figure \ref{xyz} shows that the inclusion $\phi_{car}(\mathbb{D})\subset \psi_C(\mathbb{D})$ holds where $\psi_C$ is defined in Theorem \ref{th5.3}(e). Hence $\Scar\subset\mathcal{S}^*_C$ and this proves (l).
\begin{figure}[h]
	\begin{center}
		\includegraphics[width=2in]{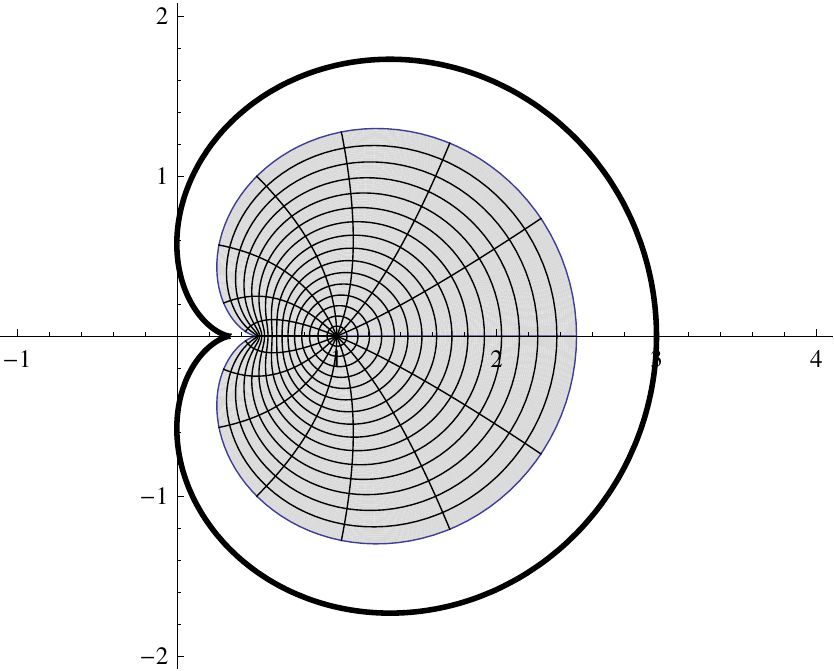}
		\caption{$\mathcal{S}^*_C$-radius for the class $\Scar$ is unity.}\label{xyz}
	\end{center}
\end{figure}

For proving (m), firstly we will show that $\Scar \subset \mathcal{M}(\beta)$ if $\beta\geq 5/2$. To see this, note that if $f \in \Scar$, Lemma \ref{result1} gives $\RE(zf'(z)/f(z))<\max\{\RE(\phi_{car}(z)):|z|=1\}=5/2$. Thus the $\mathcal{M}(\beta)$-radius of the class $\Scar$  is 1 if $\beta \in [5/2,\infty)$. Again, \lemref{result1} implies that $\RE(zf'(z)/f(z))<1+r+r^2/2<\beta$ provided $r<-1+\sqrt{2\beta-1}:=s_{12}$ where $1<\beta<5/2.$ The function $f_{car}$ satisfies $zf_{car}'(z)/f_{car}(z)=\beta$ at $z=s_{12}$.
\end{proof}

\section{$\mathcal{S}^*_{car}$-radius for ratio functions}
Let $R_i^\chi$ ($i=1,2,3$) denote the $\Scar$-radius of the classes $\mathcal{F}_i^\chi$ ($i=1,2,3$) introduced in the first section. In this section, we will compute these radii for several choices of $\chi$. The proof of the result makes use of the fact that an analytic function $p(z)=1+c_1 z+c_2 z^2+\cdots$ with $\RE p(z)>\alpha$ for all $z\in \mathbb{D}$ satisfies
\[\left|\frac{zp'(z)}{p(z)}\right|\leq \frac{2 r(1-\alpha)}{(1-r)(1+(1-2\alpha)r)}\]
for $|z|=r<1$, where $0\leq \alpha<1$. The class of such functions is denoted by $\mathcal{P}(\alpha)$ and set $\mathcal{P}:=\mathcal{P}(0)$. Also, the transformations $1/(1+z)$, $(1+z)/(1-z)$ and $(z+1)/(z+2)$ maps the disk $|z|\leq r$ onto the disks
\[	 \left|\frac{1}{1+z}-\frac{1}{1-r^2}\right|\leq\frac{r}{1-r^2},\quad \quad\left|\frac{1+z}{1-z}-\frac{1+r^2}{1-r^2}\right|\leq \frac{2r}{1-r^2}\]
and
\[\left|\frac{z+1}{z+2}-\frac{2-r^2}{4-r^2}\right|\leq \frac{r}{4-r^2}\]
respectively.

\begin{theorem}\label{th7.1}
The $\Scar$-radius of the classes $\mathcal{F}_i^\chi$ ($i=1,2,3$) is given by the following table:
\begin{center}
\begin{tabular}{QQQQQ}
\toprule
S. No.& \qquad\chi & \qquad R_1^\chi &\qquad R_2^\chi & \qquad R_3^\chi \\\midrule
(a) & \qquad z & \quad\displaystyle\frac{1}{4+\sqrt{17}} &\quad\displaystyle\frac{1}{3+2\sqrt{3}} & \quad\displaystyle\frac{1}{2+\sqrt{5}} \\\midrule
(b) & \quad \displaystyle\frac{z}{1+z} & \quad 5-2\sqrt{6} &\quad \sqrt{17}-4 &\quad 3-2\sqrt{2} \\\midrule
(c) & \quad\displaystyle\frac{z}{1-z^2} & \quad r_1\approx 0.11667 &\quad r_2\approx0.14326 &\quad  r_3\approx0.20213 \\\midrule
(d) & \quad\displaystyle\frac{z}{(1-z)^2}& \quad\displaystyle\frac{6-\sqrt{33}}{3}& \quad 5-2\sqrt{6} & \quad \displaystyle\frac{4-\sqrt{13}}{3}\\\midrule
(e) & \quad z+\displaystyle\frac{z^2}{2}&\quad s_1\approx0.10924 &\quad s_2\approx0.13414 &\quad s_3\approx0.19028\\\bottomrule
\end{tabular}
\end{center}
Here $r_1$, $r_2$, $r_3$ are the smallest positive real root of the equations $3r^4-8r^3-4r^2-8r+1=0$, $r^4-6r^3-6r^2-6r+1=0$ and $ 3r^4-4r^3-4r^2-4r+1=0$ respectively in $(0,1)$. Also $s_1$, $s_2$, $s_3$ are the smallest positive real root of the equations $3r^3+6r^2-19r+2=0$, $5r^3-15r+2=0$ and $3r^3+2r^2-11r+2=0$ respectively in $(0,1)$. All the estimates are sharp.
\end{theorem}

\begin{proof}
(a) Let $\chi(z)=z$. Firstly, suppose that $f\in\mathcal{F}_1^\chi$ and let us define functions $p,h:\disc\rightarrow\mathbb{C}$ by $p(z)=g(z)/z$ and $h(z)=f(z)/g(z)$. By definition of the class $\mathcal{F}_1^\chi$, $p,h\in\mathcal{P}$. Since $f(z)=zp(z)h(z),$ we have
	\[\frac{zf'(z)}{f(z)}=1+\frac{zh'(z)}{h(z)}+\frac{zp'(z)}{p(z)}\]
so that
	\[\left|\frac{zf'(z)}{f(z)}-1\right|\leq \frac{4r}{1-r^{2}}. \]
The above disk lies in the domain $\Omega_{car}$ if $4r/(1-r^{2})\leq 1/2$ by Lemma \ref{result7}. This gives $r^{2}+8r-1\leq 0$ which yields $r\leq 1/(4+\sqrt{17}):=R_1^\chi\approx0.1231.$	The result is sharp for $f_0(z)=z(1+z)^2/(1-z)^2 $ with the corresponding function $g_0(z)=z(1+z)/(1-z).$ Observe that $f_0(z)/g_0(z)=g_0(z)/z=(1+z)/(1-z)\in\mathcal{P}.$ Hence $f\in\mathcal{F}_1^\chi$ and $zf'(z)/f(z)$ equals $1/2$ at $z=-R_1^\chi$.
	
Secondly, suppose that $f\in\mathcal{F}_2^\chi$ and let us define functions $p,h:\disc\rightarrow\mathbb{C}$ as $p(z)=g(z)/z$ and $h(z)=g(z)/f(z).$ Then $p\in\mathcal{P}$, $h\in\mathcal{P}(1/2)$ and
	\[\frac{zf'(z)}{f(z)}=1+\frac{zp'(z)}{p(z)}-\frac{zh'(z)}{h(z)},\]
so that
	\[\left|\frac{zf'(z)}{f(z)}-1\right|\leq \frac{3r+r^{2}}{1-r^{2}}.\]
\lemref{result7} shows that this disk lies in  $\Omega_{car}$ provided $r\leq 1/(3+2\sqrt{3}):=R_2^\chi\approx 0.154701.$ The result is sharp  for the function $f_0(z)=z(1+z)^2/(1-z)$ with $g_0(z)=z(1+z)/(1-z)$. It is easy to see that $f_0 \in \mathcal{F}_2^\chi$ and
	 \[\left.\frac{zf'_0(z)}{f_0(z)}\right|_{z=-R_2^\chi}=\left.\frac{1+3z-2z^2}{1-z^2}\right|_{z=-R_2^\chi}=\frac{1}{2}.\]
	
Next, let $f\in\mathcal{F}_3^\chi$. Then the function $h(z)=f(z)/z$ belongs to $\mathcal{P}$ and a straightforward calculation shows that
\[\left|\frac{zf'(z)}{f(z)}-1\right|=\left|\frac{zh'(z)}{h(z)}\right|\leq \frac{2r}{1-r^2}.\]
Again \lemref{result7} implies that $2r/(1-r^2)\leq1/2$ which is equivalent to $r\leq 1/(2+\sqrt{5}):=R_3^\chi\approx 0.23606.$ For sharpness, consider the function $f_0(z)=z(1+z)/(1-z)$ which satisfies $zf_0'(z)/f_0(z)=1/2$ at $z=-R_3^\chi$.

(b) Let $\chi(z)=z/(1+z)$. For $f\in\mathcal{F}_1^\chi$, define the functions $p,h:\disc\rightarrow\mathbb{C}$ as $p(z)=f(z)/g(z)$ and $h(z)=(1+z)g(z)/z$. Then $p,h\in\mathcal{P}$ and $f(z)=zp(z)h(z)/(1+z).$ A simple calculation shows that
	 \[\frac{zf'(z)}{f(z)}=\frac{zh'(z)}{h(z)}+\frac{zp'(z)}{p(z)}+\frac{1}{1+z} \]
so that
\[\left|\frac{zf'(z)}{f(z)}-\frac{1}{1-r^2}\right|\leq \frac{5r}{1-r^2}. \]
As pointed out in \cite{ASHA}, the radius of starlikeness of the class $\mathcal{F}_1^\chi$ is $1/5$ and since $\Scar\subset \mathcal{S}^*$, we must have $r\leq 1/5$ so that the centre of the disk satisfies $1<1/(1-r^2)<3/2$. By applying \lemref{result7}, we have
	 \[\frac{5r}{1-r^2}\leq \frac{1}{1-r^2}-\frac{1}{2}\]
This is possible if $r\leq5-2\sqrt{6}:=R_1^\chi\approx 0.10102.$ For the function $f_0(z)=z(1-z)^2/(1+z)^3$ with $g_0(z)=z(1-z)/(1+z)^2$, note that $f_0\in  \mathcal{F}_1^\chi$ and $zf_0'(z)/f_0(z)=1/2$ at $z=R_1^\chi$.

	If $f \in\mathcal{F}_2^\chi$ and if we define the functions $h,p:\disc\rightarrow\mathbb{C}$ by $h(z)=g(z)/f(z)$ and $p(z)=(1+z)g(z)/z$, then $h\in\mathcal{P}(1/2),$  $p\in\mathcal{P}$ and $f(z)=zp(z)/((1+z)h(z))$ so that
	 \[\frac{zf'(z)}{f(z)}=\frac{zp'(z)}{p(z)}-\frac{zh'(z)}{h(z)}+\frac{1}{1+z}\]
which yields
	 \[\left|\frac{zf'(z)}{f(z)}-\frac{1}{1-r^2}\right|\leq\frac{4r+r^2}{1-r^2}. \]
The radius of starlikeness of the class $\mathcal{F}_2^\chi$ is $\sqrt{5}-2$ (see \cite{ASHA}). As a result, the centre of the above disk is less than $3/2$ and the  disk lies in $\Omega_{car}$ provided
	 \[\frac{4r+r^2}{1-r^2}\leq\frac{1}{1-r^2}-\frac{1}{2}\]
by \lemref{result7}. This gives $r\leq \sqrt{17}-4:=R_2^\chi\approx 0.12310.$ This bound is best possible for the function $f_0(z)=z(1-z)^2/(1+z)^2$ and $g_0(z)=z(1-z)/(1+z)^2.$ Clearly, $f_0\in\mathcal{F}_2^\chi$ and $zf_0'(z)/f_0(z)$ assumes the value $1/2$ at $z=R_2^\chi$.

	 If $f\in\mathcal{F}_3^\chi$, then the function $h:\disc\rightarrow\mathbb{C}$ defined as $h(z)=(1+z)f(z)/z$ satisfies $h\in\mathcal{P}$ and $f(z)=zh(z)/(1+z)$. Using the fact that the radius of starlikeness of the class $\mathcal{F}_3^\chi$ is $1/3$ (see \cite{ASHA}), the disk
\[\left|\frac{zf'(z)}{f(z)}-\frac{1}{1-r^2}\right|\leq\frac{3r}{1-r^2}\]
lies inside $\Omega_{car}$ for $r\leq3-2\sqrt{2}:=R_3^\chi\approx 0.17157$ as a consequence of Lemma \ref{result7}. Sharpness of the result is evident by taking function $f_0(z)=z(1-z)/(1+z)^2$ for which $zf_0'(z)/f_0(z)=1/2$ at $z=R_3^\chi$.
	
	(c) Let $\chi(z)=z/(1-z^2)$. Let $f\in\mathcal{F}_1^\chi$ and define the functions $q_1,q_2:\disc\rightarrow\mathbb{C}$ as $q_1(z)=f(z)/g(z)$ and $q_2(z)=(1-z^2)g(z)/z$. Then $q_1,q_2\in\mathcal{P}$ and $f(z)=zq_1(z)q_2(z)/(1-z^2).$ A routine calculation shows that
\[\frac{zf'(z)}{f(z)}=\frac{zq_1'(z)}{q_1(z)}+\frac{zq_2'(z)}{q_2(z)}+\frac{1+z^2}{1-z^2}\]
so that
\[\left|\frac{zf'(z)}{f(z)}-\frac{1+r^4}{1-r^4}\right|\leq \frac{2r(2r^2+r+2)}{1-r^4}.\]
In \cite{KHATTER2}, it has been shown that the radius of starlikeness of the class $\mathcal{F}_1^\chi$ is $1+\sqrt{2}-(2(1+\sqrt{2}))^{1/2}\approx 0.216845:=u_0$. As a result, $r\leq u_0$ and the centre of the disk $(1+r^4)/(1-r^4)\leq (1+u_0^4)/(1-u_0^4)\approx 1.00443$. Using these observations and \lemref{result7}, we obtain
\[\frac{2r(2r^2+r+2)}{1-r^4}\leq \frac{1+r^4}{1-r^4}-\frac{1}{2}. \]
This yields $r\leq r_1\approx0.116675$ where $r_1$  is the smallest positive real root of the equation $3r^4-8r^3-4r^2-8r+1=0$ in $(0,u_0)$. The result is sharp for the function $f_0\in\mathcal{F}_1^\chi$ given by
\[f_0(z)=\frac{z(1+iz)^2}{(1-z^2)(1-iz)^2}\quad \mbox{with} \quad g_0(z)=\frac{z(1+iz)}{(1-z^2)(1-iz)}.\]
Note that $zf_0'(z)/f_0(z)=(z^4-4iz^3+2z^2+4iz+1)/(1-z^4)$ and it assumes the value $1/2$ at $z=ir_1$.

For $f\in\mathcal{F}_2^\chi$, define the functions $q_1,q_2:\disc\rightarrow\mathbb{C}$ by $q_1(z)=g(z)/f(z)$ and $q_2(z)=(1-z^2)g(z)/z$. Then $p_1\in\mathcal{P}(1/2)$, $p_2\in\mathcal{P}$ and $f(z)=zq_2(z)/(q_1(z)(1-z^2))$ so that
\[\frac{zf'(z)}{f(z)}=\frac{zq_2'(z)}{q_2(z)}-\frac{zq_1'(z)}{q_1(z)}+\frac{1+z^2}{1-z^2}.\]
This gives
\[\left|\frac{zf'(z)}{f(z)}-\frac{1+r^4}{1-r^4}\right|\leq \frac{r(r^3+3r^2+3r+3)}{1-r^4}.\]
As observed in \cite{KHATTER2}, the radius of starlikeness of the class $\mathcal{F}_2^\chi$ is
\[v_0:=\frac{1}{3}\left( -1 - \frac{2^{2/3}}{(4 + 3 \sqrt{2})^{1/3}} + (2 (4 + 3 \sqrt{2}))^{1/3}\right)\approx0.253077\]
so that the centre of the disk $(1+r^4)/(1-r^4)\leq (1+v_0^4)/(1-v_0^4)\approx 1.00824$. Thus
\[\frac{r(r^3+3r^2+3r+3)}{1-r^4}\leq\frac{1+r^4}{1-r^4}-\frac{1}{2}.\]
This holds for $r\leq r_2$ where
\[r_2=\frac{3}{2}+\frac{\sqrt{17}}{2}-\sqrt{\frac{1}{2}(11+3\sqrt{17})}\approx 0.14327\]
is the smallest positive root of the equation $r^4-6r^3-6r^2-6r+1=0$ in $(0,v_0)$. For sharpness of this bound, consider the function $f_0\in\mathcal{F}_2^\chi$ given by
\[f_0(z)=\frac{z(1+iz)^2}{(1-z^2)(1-iz)}\quad \mbox{with}\quad g_0(z)=\frac{z(1+iz)}{(1-z^2)(1-iz)}.\]
It is easily seen that $zf_0'(z)/f_0(z)=(1+3iz+3z^2-3iz^3)/(1-z^4)=1/2$ at $z=i r_2$.

Next, suppose that $f \in\mathcal{F}_3^\chi.$ The function $p:\disc\rightarrow\mathbb{C}$ defined as $p(z)=(1-z^2)f(z)/z$ belongs to the class $\mathcal{P}$ and it is easy to deduce that
\[\left|\frac{zf'(z)}{f(z)}-\frac{1+r^4}{1-r^4}\right|\leq \frac{2r(r^2+r+1)}{1-r^4}.\]
The radius of starlikeness of the class $\mathcal{F}_3^\chi$ (see \cite{KHATTER2}) is
\[w_0:=\frac{1}{2}+\frac{\sqrt{5}}{2}-\sqrt{\frac{1}{2}(1+\sqrt{5})}\approx0.346014.\]
This shows that the centre $(1+r^4)/(1-r^4)\leq (1+w_0^4)/(1-w_0^4)\approx 1.02909$ and Lemma \ref{result7} provides the inequality
\[\frac{2r(r^2+r+1)}{1-r^4}\leq \frac{1+r^4}{1-r^4}-\frac{1}{2}\]
which is possible if $r\leq r_3\approx 0.202135$ where $r_3$ is the smallest positive root of the equation $3r^4-4r^3-4r^2-4r+1=0$ in $(0,w_0)$. The function $f_0(z)=z(1+iz)/((1-z^2)(1-iz))\in \mathcal{F}_3^\chi$ and at $z=ir_3$, $zf_0'(z)/f_0(z)=(z^4-2iz^3+2z^2+2iz+1)/(1-z^4)=1/2$.

	(d) Let $\chi(z)=k(z)=z/(1-z)^2$ be the Koebe function. For $f\in\mathcal{F}_1^\chi$, define functions $p_1,p_2:\disc\rightarrow\mathbb{C}$ by $q_1(z)=f(z)/g(z)$ and $q_2(z)=g(z)/k(z).$ The function $q_1,q_2\in\mathcal{P}$ and $f(z)=q_1(z)q_2(z)k(z)$ which gives
	\[\frac{zf'(z)}{f(z)}=\frac{zq_1'(z)}{q_1(z)}+\frac{zq_2'(z)}{q_2(z)}+\frac{1+z}{1-z}.\]
Thus we have
	\[\left|\frac{zf'(z)}{f(z)}-\frac{1+r^2}{1-r^2}\right|\leq \frac{6r}{1-r^2}.\]
El-Faqeer \emph{et al.} \cite{AMVS} determined the radius of starlikeness of the class $\mathcal{F}_1^\chi$ which turn out to be $3-2\sqrt{2}\approx 0.171573$. Consequently $r\leq 3-2\sqrt{2}$ and the centre $(1+r^2)/(1-r^2)\leq 3/(2\sqrt{2})\approx 1.06066$. In view of \lemref{result7}, it follows that
	\[\frac{6r}{1-r^2}\leq \frac{1+r^2}{1-r^2}-\frac{1}{2}.\]
This is true for $r\leq (6-\sqrt{33})/3:=R_1^\chi\approx 0.0851458$. The result is sharp for the function $f_0(z)=z(1+z)^2/(1-z)^4$ with $g_0(z)=z(1+z)/(1-z)^3$. For $z=-R_1^\chi$, $zf_0'(z)/f_0(z)=(z^2+6z+1)/(1-z^2)=1/2$.

Next, let $f\in\mathcal{F}_2^\chi$. Define functions $q_1,q_2:\disc\rightarrow\mathbb{C}$ as $q_1(z)=g(z)/f(z)$ and $q_2(z)=g(z)/k(z)$. Then $p_1\in\mathcal{P}(1/2)$, $p_2\in\mathcal{P}$ and
	\[\frac{zf'(z)}{f(z)}=\frac{zp_2'(z)}{p_2(z)}-\frac{zp_1'(z)}{p_1(z)}+\frac{1+z}{1-z}\]
so that
\[\left|\frac{zf'(z)}{f(z)}-\frac{1+r^2}{1-r^2}\right|\leq \frac{5r+r^2}{1-r^2}.\]
Since the radius of starlikeness of the class $\mathcal{F}_2^\chi$ is $1/5$ (see \cite{AMVS}), it is easily seen that $(1+r^2)/(1-r^2)\leq 13/12\approx 1.08333$. The domain $\Omega_{car}$ contains the above disk provided
	\[\frac{5r+r^2}{1-r^2}\leq \frac{1+r^2}{1-r^2}-\frac{1}{2}\]
by Lemma \ref{result7}, which is possible for $r\leq 5-2\sqrt{6}\approx 0.101021.$ This bound is sharp by considering the function $f_0(z)=z(1+z)^2/(1-z)^3$ with $g_0(z)=z(1+z)/(1-z)^3$. Observe that $zf_0'(z)/f_0(z)=(1+5z)/(1-z^2)=1/2$ at $z=-R_2^\chi$.

 Now, suppose that $f\in\mathcal{F}_3^\chi$ and let $h(z)=f(z)/k(z)$. Then $h\in\mathcal{P}$ and a simple computation gives
	 \[\left|\frac{zf'(z)}{f(z)}-\frac{1+r^2}{1-r^2}\right|\leq \frac{4r}{1-r^2}.\]
	 Ratti \cite[Theorem 3, p.\ 244]{RATTI} proved that the radius of starlikeness of the class $\mathcal{F}_3^\chi$ is $2-\sqrt{3}$ (in a more general setup). As a result, $r\leq 2-\sqrt{3}$ and $(1+r^2)/(1-r^2)\leq 2/\sqrt{3}\approx 1.1547$. \lemref{result7} shows that
\[\frac{4r}{1-r^2}\leq \frac{1+r^2}{1-r^2}-\frac{1}{2}\]
which provides $r\leq (4-\sqrt{13})/3\approx 0.13148.$ The result is sharp for the function $f_0(z)=z(1+z)/(1-z)^3.$

	(e) Let $\chi(z)=z+z^2/2$. If $f\in\mathcal{F}_1^\chi$, define functions $q_1,q_2:\disc\rightarrow\mathbb{C}$ as $q_1(z)=f(z)/g(z)$ and $q_2(z)=g(z)/(z+z^2/2)$. Then $q_1,q_2\in\mathcal{P}$ and $f(z)=q_1(z)q_2(z)\left(z+z^2/2\right).$ A calculation yields
	\[\frac{zf'(z)}{f(z)}=\frac{zq_1'(z)}{q_1(z)}+\frac{zq_2'(z)}{q_2(z)}+\frac{2(z+1)}{z+2}.\]
	Thus we have
	 \[	 \left|\frac{zf'(z)}{f(z)}-\frac{4-2r^2}{4-r^2}\right|\leq \frac{6r(3-r^2)}{(1-r^2)(4-r^2)}.\]
The function $\psi(r)=(4-2r^2)/(4-r^2)$ is a decreasing function of r in the interval $(0,1)$. As a result, the centre of the disk satisfies  $2/3<(4-2r^2)/(4-r^2)<1$. By \lemref{result7}, the disk lies in $\Omega_{car}$ if
	 \[\frac{6r(3-r^2)}{(1-r^2)(4-r^2)}\leq \frac{4-2r^2}{4-r^2}-\frac{1}{2}\]
which holds for $r\leq s_1$ where $s_1$ is the smallest positive root of the equation $3r^3+6r^2-19r+2=0$ in $(0,1)$. The result is sharp for the function
	 \[f_0(z)=\frac{(1+z)^2(z+z^2/2)}{(1-z)^2}\quad \mbox{with}\quad g_0(z)=\frac{(1+z)(z+z^2/2)}{(1-z)}.\]
Clearly $f_0\in \mathcal{F}_1^\chi$ and
	 \[\left.\frac{zf_0'(z)}{f_0(z)}\right|_{z=-s_1}=\frac{2(1-5s_1+s_1^2+s_1^3)}{(2-s_1)(1-s_1^2)}=\frac{1}{2}.\]
	
 Next, let $f \in \mathcal{F}_2^\chi$ and $q_1,q_2:\disc\rightarrow\mathbb{C}$ be analytic functions defined as $q_1(z)=g(z)/f(z)$ and $q_2(z)=g(z)/(z+z^2/2)$. Then $q_1\in\mathcal{P}(1/2)$ and $q_2\in\mathcal{P}.$ Also, $f(z)=(z+z^2/2)q_2(z)/q_1(z)$ which leads to the expression
\[\frac{zf'(z)}{f(z)}=\frac{zq_2'(z)}{q_2(z)}-\frac{zq_1'(z)}{q_1(z)}+\frac{2(z+1)}{z+2}.\]
By Lemma \ref{result7}, domain $\Omega_{car}$ contains the disk
\[\left|\frac{zf'(z)}{f(z)}-\frac{4-2r^2}{4-r^2}\right|\leq \frac{r(-r^3-5r^2+4r+14)}{(1-r^2)(4-r^2)}\]
provided
\[\frac{r(-r^3-5r^2+4r+14)}{(1-r^2)(4-r^2)}\leq \frac{4-2r^2}{4-r^2}-\frac{1}{2}.\]
This inequality is satisfied for $r\leq s_2\approx 0.134138$ where $s_2$ is the smallest positive real root of the equation $5r^3-15r+2=0$ in $(0,1)$. The result is sharp for the function $f_0\in \mathcal{F}_2^\chi$ defined as
\[f_0(z)=\frac{(1+z)^2(z+z^2/2)}{1-z}\quad \mbox{with}\quad g_0(z)=\frac{(1+z)(z+z^2/2)}{(1-z)}\]
as it satisfies
\[\left.\frac{zf_0'(z)}{f_0(z)}\right|_{z=-s_2}=\frac{2-8s_2-s_2^2+3s_2^3}{(2-s_2)(1-s_2^2)}=\frac{1}{2}.\]

Finally, let $f\in\mathcal{F}_3^\chi$. The function $h:\disc\rightarrow\mathbb{C}$ defined as $h(z)=f(z)/(z+z^2/2)$ belongs to class $\mathcal{P}$ and a simple computation leads to the disk
\[\left|\frac{zf'(z)}{f(z)}-\frac{4-2r^2}{4-r^2}\right|\leq \frac{2r(5-2r^2)}{(1-r^2)(4-r^2)}\]
which lies in the domain $\Omega_{car}$ if
\[\frac{2r(5-2r^2)}{(1-r^2)(4-r^2)}\leq \frac{4-2r^2}{4-r^2}-\frac{1}{2}.\]
This yields $r\leq s_3\approx 0.19028$, where $s_3$ is the smallest positive root of the equation $3r^3+2r^2-11r+2=0.$ The result is sharp for the function $f_0(z)=(1+z)(z+z^2/2)/(1-z)$.
\end{proof}

Let us provide an application of Theorem \ref{th7.1}. Motivated by Ling and Ding \cite{LING}, Ali \emph{et al.} \cite{ALI2} and Jain and Ravichandran \cite{NAVEEN} determined the bounds on $\rho$ such that the function $h_\rho(z)=(f\ast g)(\rho z)/\rho$ belongs to various subclasses of $\mathcal{S}^*$ if $f, g \in \mathcal{S}^*$. Using the similar technique, we will prove this result for the class $\Scar.$
\begin{corollary}
	Let $f,g\in\mathcal{S}^*.$ Then $h_\rho(z)=(f\ast g)(\rho z)/\rho\in\Scar$ for $0< \rho \leq (4-\sqrt{13})/3\approx 0.1314829$. The radius is sharp.
\end{corollary}
\begin{proof}
By the proof of Theorem \ref{th7.1}(d), the $\Scar$-radius of the class $\mathcal{F}_3^{z/(1-z)^2}$ is $\rho_0=(4-\sqrt{13})/3$ and this radius is sharp for the function $f_0(z)=z(1+z)/(1-z)^3$. Since $f,g\in\mathcal{S}^*$, the functions $F$ and $G$ defined by $zF'(z)=f(z)$ and $zG'(z)=g(z)$ belongs to the class $\mathcal{K}.$ Consequently $F\ast G\in\mathcal{K}.$ By using Lemma \ref{result10}, $(f_0*F*G)(\rho z)/\rho\in\Scar$ for $0<\rho\leq \min\{1/2, \rho_0\}=\rho_0$. Since $f\ast g=f_0*F\ast G$, the result follows.
\end{proof}

\section*{Acknowledgements}
The first author is supported by a Junior Research Fellowship from Council of Scientific and Industrial Research (CSIR), New Delhi with File No. 09/045(1727)/2019-EMR-I.

\end{document}